\documentclass{article}%

\usepackage{amsmath,amssymb,amsfonts}%
\usepackage{amscd}%
\usepackage{theorem}%
\usepackage{enumerate}%
\usepackage[arrow,matrix,tips]{xy}%
\usepackage{hyperref}%

\theoremstyle{change}%
\newtheorem{definition}{Definition:}[section]%
\newtheorem{theorem}[definition]{Theorem:}%
\newtheorem{proposition}[definition]{Proposition:}%
\newtheorem{lemma}[definition]{Lemma:}%
\newtheorem{lemma_txt}[definition]{Lemma}%
\newtheorem{corollary}[definition]{Corollary:}%
{\theorembodyfont{\rmfamily} \newtheorem{remark}[definition]{Remark:}}%
{\theorembodyfont{\rmfamily} \newtheorem{example}[definition]{Example:}}%

\newenvironment{proof}
  {{\bf Proof:}}
  {\qquad \hspace*{\fill} $\Box$}%

\newcommand{\tm}{\times}%
\newcommand{\cl}{\operatorname{cl}}%
\newcommand{\ep}{\varepsilon}%
\newcommand{\diam}{\operatorname{diam}}%
\newcommand{\id}{\operatorname{id}}%
\newcommand{\Lip}{\operatorname{Lip}}%
\newcommand{\Mis}{\operatorname{M}}%

\newcommand{\N}{\mathbb{N}}%
\newcommand{\R}{\mathbb{R}}%
\newcommand{\Z}{\mathbb{Z}}%

\newcommand{\AC}{\mathcal{A}}%
\newcommand{\BC}{\mathcal{B}}%
\newcommand{\CC}{\mathcal{C}}%
\newcommand{\DC}{\mathcal{D}}%
\newcommand{\EC}{\mathcal{E}}%
\newcommand{\LC}{\mathcal{L}}%
\newcommand{\NC}{\mathcal{N}}%
\newcommand{\PC}{\mathcal{P}}%
\newcommand{\QC}{\mathcal{Q}}%
\newcommand{\UC}{\mathcal{U}}%
\newcommand{\VC}{\mathcal{V}}%

\newcommand{\rme}{\mathrm{e}}%
\newcommand{\rmd}{\mathrm{d}}%
\newcommand{\rmS}{\mathrm{S}}%
\newcommand{\rmD}{\mathrm{D}}%

\newcommand{\tp}{\operatorname{top}}%
\newcommand{\sep}{\operatorname{sep}}%
\newcommand{\spn}{\operatorname{span}}%

\begin{document}

\title{Expanding and expansive time-dependent dynamics}%
\author{Christoph Kawan\footnote{Universit\"{a}t Passau, Fakult\"{a}t f\"{u}r Informatik und Mathematik, Innstra{\ss}e 33, 94032 Passau (Germany); e-mail: christoph.kawan@uni-passau.de}}%
\date{}%
\maketitle%

\begin{abstract}
In this paper, time-dependent dynamical systems given by sequences of maps are studied. For systems built from expanding $\CC^2$-maps on a compact Riemannian manifold $M$ with uniform bounds on expansion factors and derivatives, we provide formulas for the metric and topological entropy. If we only assume that the maps are $\CC^1$, but act in the same way on the fundamental group of $M$, we can show the existence of an equi-conjugacy to an autonomous system, implying a full variational principle for the entropy. Finally, we introduce the notion of strong uniform expansivity that generalizes the classical notion of positive expansivity, and we prove time-dependent analogues of some well-known results. In particular, we generalize Reddy's result which states that a positively expansive system  locally expands distances in an equivalent metric.%
\end{abstract}


\section{Introduction}%

Uniformly expanding maps are the simplest non-trivial examples of discrete-time dynamical systems within the theory of finite-dimensional differentiable systems. From today's perspective their analysis can be seen as the starting point of a long and fruitful thread of research in differentiable dynamics, with different stages of generalization, from expanding to uniformly hyperbolic, to non-uniformly or partially hyperbolic. Recently, there have been major efforts in establishing a general theory of systems with time-dependent dynamical laws, often called nonautonomous, sequential or non-stationary dynamical systems, see for instance \cite{BVa,Kaw,KRa,KSn,KMS,OSY}. Instead by the iteration of one map, a discrete-time nonautonomous system is defined by a sequence of maps which are composed in the given order, i.e., at each time instant the dynamical law can be different. Though the range of phenomena to be observed in such systems is certainly much broader than in the autonomous case, the study of smooth uniformly expanding systems might as well be a good starting point here.%

The papers \cite{LYo} by Lasota and Yorke and \cite{OSY} by Ott, Stenlund and Young laid the foundations for the study of statistical properties of systems defined by the composition of $\CC^2$-expanding maps $f_n:M\rightarrow M$ on a compact Riemannian manifold $M$. Here the main focus is on (exponential) loss of memory, the time-dependent analogue of decay of correlations. In particular, in \cite{OSY} the authors prove that two positive initial densities with respect to the Riemannian volume measure converge to each other in the $L^1$-sense at an exponential rate under the evolution of the nonautonomous system, although none of them necessarily tends to a limit. For this strong result some further restrictions on the sequence of expanding maps is necessary, namely uniform bounds on the expansion factors and the first and second derivatives. Without such uniform bounds, the convergence may not be exponential on the whole positive time axis.%

Other classical concepts that have been extended to nonautonomous systems are those of topological and measure-theoretic entropy, cf.~\cite{Can,Kaw,KSn,KMS}. In particular, in \cite{Kaw} one part of the variational principle for entropy was established under quite general conditions. In the first part of this paper, we extend the notions of topological and measure-theoretic pressure to nonautonomous systems and prove that the second is always bounded by the first, generalizing the corresponding inequality for the entropies. We use the distortion lemma from \cite{OSY} to prove a Bowen-Ruelle-type volume lemma first, which together with the inequality of pressures then yields a formula for the metric entropy of an expanding nonautonomous system with respect to smooth initial measures. Subsequently, we also provide a formula for the topological entropy, using again the volume lemma.%

For a nonautonomous system built from $\CC^1$-expanding maps $f_n:M\rightarrow M$ that all act in the same way on the fundamental group of $M$, we generalize a classical result of Shub \cite{Shu}. We prove the existence of an equi-conjugacy to an autonomous system $f:M\rightarrow M$, i.e., of a sequence $(\pi_n)$ of homeomorphisms such that $\pi_{n+1} \circ f_n \equiv f \circ \pi_n$, where both $\{\pi_n\}$ and $\{\pi_n^{-1}\}$ are equicontinuous families. Using that equi-conjugacies preserve the topological as well as the metric entropy, we can conclude a full variational principle for such systems.%

In the last part of the paper, we introduce for a topological nonautonomous system the notion of strong uniform expansivity, which generalizes the classical notion of positive expansivity. In particular, the $\CC^2$-expanding systems considered in the first part of the paper satisfy this property. Conversely, we show that strongly uniformly expansive systems admit uniformly equivalent metrics in which distances are expanded locally, and uniformly with respect to the initial time. This generalizes a classical result of Reddy \cite{Red}.%

The paper is organized as follows. In Section \ref{sec_prelim}, we give an overview of the main concepts and introduce notation. Section \ref{sec_pressure} introduces the notions of topological and measure-theoretic pressure and contains the proof of the variational inequality. In Section \ref{sec_indep}, we show that the metric entropy of an expanding nonautonomous system does not depend on the initial measure as long as it has a positive Lipschitz density with respect to the Riemannian volume. Subsequently, in Section \ref{sec_entforms} we provide a formula for the metric entropy as well as for the topological entropy. Section \ref{sec_equiconj} contains the proof of the time-dependent conjugacy result for $\CC^1$-expanding systems. Finally, in Section \ref{sec_sue} the notion of strong uniform expansivity is defined and several properties are shown.%

\section{Preliminaries}\label{sec_prelim}%

We write $\N$ for the set of positive integers, $\N_0 = \{0\}\cup\N$, and $\R$ for the set of real numbers. For a finite set $A$, $\# A$ denotes the number of elements in $A$. We write $B_{\ep}(x)$ or $B_{\ep}(x;d)$ for the open $\ep$-ball around $x$ in a metric space $(X,d)$. A \emph{nonautonomous dynamical system}, or an \emph{NDS} for short, is a pair $(X_{0,\infty},f_{0,\infty})$, where $X_{0,\infty} = (X_n)_{n=0}^{\infty}$ is a sequence of sets and $f_{0,\infty} = (f_n)_{n=0}^{\infty}$ is a sequence of maps $f_n:X_n \rightarrow X_{n+1}$. If all the sets $X_n$ are compact metric spaces with associated metrics $d_n$, and all the $f_n$ are continuous, we speak of a \emph{topological NDS}. If the sets $X_n$ are probability spaces with associated $\sigma$-algebras $\AC_n$ and probability measures $\mu_n$, such that the $f_n$ are measurable and satisfy $f_n\mu_n \equiv \mu_{n+1}$ (where $f_n$ here stands for the induced push-forward operator on measures), we speak of a \emph{measure-theoretic} or \emph{metric NDS}. Then we call $\mu_{0,\infty} := (\mu_n)_{n=0}^{\infty}$ an \emph{invariant measure sequence} or an \emph{IMS} for $(X_{0,\infty},f_{0,\infty})$.%

The time evolution of the system is defined by composing the maps $f_n$ in the obvious way. In general, we define%
\begin{equation*}
  f_k^n := f_{k+n-1} \circ \cdots \circ f_{k+1} \circ f_k \quad \mbox{for } k \in \N_0,\ n\in\N,\qquad f_k^0 := \id_{X_k}.%
\end{equation*}
We also put $f_k^{-n} := (f_k^n)^{-1}$, which is only applied to sets. We write $(X_{i,\infty},f_{i,\infty})$ for the pair of shifted sequences $(X_i,X_{i+1},\ldots)$ and $(f_i,f_{i+1},\ldots)$, respectively. Moreover, we abbreviate $(X_{0,\infty},f_{0,\infty})$ by $(X_{\infty},f_{\infty})$, and we use analogous notation for other sequences of objects related to an NDS. If such a sequence is constant, e.g., $X_n \equiv X$, we simply write $X$ instead of $X_{\infty}$.%

We recall the definition of metric entropy for a metric NDS $(X_{\infty},f_{\infty})$. If $\PC_{\infty} = (\PC_n)_{n\geq0}$ is a sequence of finite measurable partitions for the spaces $X_n$,%
\begin{equation*}
  h(f_{\infty};\PC_{\infty};\mu_{\infty}) := \limsup_{n\rightarrow\infty}\frac{1}{n}H_{\mu_0}\left(\bigvee_{i=0}^{n-1}f_0^{-i}\PC_i\right),%
\end{equation*}
where $H_{\mu_0}(\cdot) = - \sum_{P \in (\cdot)}\mu_0(P)\log \mu_0(P)$ denotes the usual entropy of a partition, is called the \emph{entropy of $f_{\infty}$ w.r.t.~$\PC_{\infty}$}. If the sequence $\mu_{\infty}$ is clear from the context, we omit this argument. A family $\EC$ of sequences of partitions is called an \emph{admissible class} if (i) for every $\PC_{\infty}\in\EC$ there is a uniform bound on $\#\PC_n$, (ii) if $\PC_{\infty}\in\EC$ and $\QC_{\infty}$ is coarser than $\PC_{\infty}$ (componentwise), then $\QC_{\infty}\in\EC$, and (iii) if $\PC_{\infty}\in\EC$ and $k\in\N$, then $\PC^{\langle k\rangle}_{\infty}$, defined by%
\begin{equation*}
  \PC^{\langle k \rangle}_n := \bigvee_{i=0}^{k-1}f_n^{-i}\PC_{i+n},%
\end{equation*}
is also in $\EC$. For an admissible class $\EC$, the \emph{metric entropy of $f_{\infty}$ w.r.t.~$\EC$} is defined by%
\begin{equation*}
  h_{\EC}(f_{\infty}) := \sup_{\PC_{\infty}\in\EC}h(f_{\infty};\PC_{\infty}).%
\end{equation*}

For a topological NDS $(X_{\infty},f_{\infty})$ we define the \emph{topological entropy of $f_{\infty}$ w.r.t.~a sequence $\UC_{\infty} = (\UC_n)_{n\geq0}$} of open covers for $X_{\infty}$ by%
\begin{equation*}
  h_{\tp}(f_{\infty};\UC_{\infty}) := \limsup_{n\rightarrow\infty}\frac{1}{n}\log\NC\left(\bigvee_{i=0}^{n-1}f_0^{-i}\UC_i\right),%
\end{equation*}
where $\NC(\cdot)$ denotes the minimal cardinality of a subcover. We denote the family of all sequences of open covers for the spaces $X_n$ with Lebesgue numbers bounded away from zero by $\LC(X_{\infty})$. Then the \emph{topological entropy of $f_{\infty}$} is%
\begin{equation*}
  h_{\tp}(f_{\infty}) := \sup_{\UC_{\infty}\in\LC(X_{\infty})}h_{\tp}(f_{\infty};\UC_{\infty}).%
\end{equation*}
Equivalent definitions in terms of $(n,\ep)$-spanning or $(n,\ep)$-separated subsets of $X_0$ can be given. The \emph{Bowen-metrics} on $X_i$ are given by%
\begin{equation*}
  d_{i,n}(x,y) := \max_{0\leq j \leq n}d_{i+j}(f_i^j(x),f_i^j(y)),\quad i,n \geq 0.%
\end{equation*}
An open ball of radius $\ep$ in the metric $d_{i,n}$ is denoted by $B_i^n(x,\ep)$ and called a \emph{Bowen-ball of order $n$}. We say that a set $F\subset X_0$ is \emph{$(n,\ep;f_{\infty})$-spanning} (or \emph{$(n,\ep)$-spanning}) if for every $x\in X_0$ there is $y\in F$ with $d_{0,n}(x,y) < \ep$. A set $E\subset X_0$ is \emph{$(n,\ep;f_{\infty})$-separated} (or \emph{$(n,\ep)$-separated}) if $d_{0,n}(x,y) \geq \ep$ holds for all $x\neq y$ in $E$. Then $r_{\spn}(n,\ep;f_{\infty})$ denotes the minimal cardinality of an $(n,\ep;f_{\infty})$-spanning set and $r_{\sep}(n,\ep;f_{\infty})$ the maximal cardinality of an $(n,\ep)$-separated set.%

Two topological NDSs $(X_{\infty},f_{\infty})$ and $(Y_{\infty},g_{\infty})$ are \emph{equi-conjugate} if there exists a sequence $\pi_{\infty} = (\pi_n)_{n\geq0}$ of homeomorphisms $\pi_n:X_n\rightarrow Y_n$ such that both $(\pi_n)$ and $(\pi_n^{-1})$ are uniformly equicontinuous and $\pi_{n+1} \circ f_n \equiv g_n \circ \pi_n$. In this case, $h_{\tp}(f_{\infty}) = h_{\tp}(g_{\infty})$ and such $\pi_{\infty}$ is called an equi-conjugacy (cf.~\cite[Thm.~B]{KSn}).%

For an equicontinuous NDS $(X_{\infty},f_{\infty})$ with an IMS $\mu_{\infty}$ of Borel probability measures, a special admissible class $\EC_{\Mis}$ is defined as follows: $\PC_{\infty}$, $\PC_n = \{P_{n,1},\ldots,P_{n,k_n}\}$, is in $\EC_{\Mis}$ iff for every $\ep>0$ there are $\delta>0$ and compact sets $K_{n,i}\subset P_{n,i}$ such that (i) $\mu_n(P_{n,i}\backslash K_{n,i}) \leq \ep$ and (ii) $\min_{1 \leq i<j \leq k_n}D(K_{n,i},K_{n,j}) \geq \delta$ for all $n\geq0$, where $D(K_{n,i},K_{n,j}) = \min_{(x,y)\in K_{n,i}\tm K_{n,j}}d_n(x,y)$. In \cite[Thm.~28]{Kaw} we proved the inequality%
\begin{equation*}
  h_{\EC_M}(f_{\infty}) \leq h_{\tp}(f_{\infty}).%
\end{equation*}

For each NDS $(X_{\infty},f_{\infty})$ and $k\geq1$, we define the \emph{$k$-th power system} $(X^{[k]}_{\infty},f^{[k]}_{\infty})$,%
\begin{equation*}
  X^{[k]}_n := X_{kn},\quad f^{[k]}_n := f_{kn}^k.%
\end{equation*}
For the corresponding metric and topological entropies, power rules hold (see \cite[Prop.~5 and Prop.~25]{Kaw}), e.g., if $f_{\infty}$ is equicontinuous, then%
\begin{equation*}
  h_{\tp}(f^{[k]}_{\infty}) = k \cdot h_{\tp}(f_{\infty}).%
\end{equation*}

Next, we introduce the class of nonautonomous systems to be studied in the present paper. Let $M$ be a connected and compact Riemannian manifold. By $d(\cdot,\cdot)$ we denote the Riemannian distance and by $m$ the Riemannian volume measure on $M$. For simplicity, we will assume $m(M)=1$, so $m$ is a probability measure. For any $\lambda>1$ and $\Gamma>\lambda$ consider the set%
\begin{equation*}
  \EC(\lambda,\Gamma) := \left\{f \in \CC^2(M,M)\ :\ f \mbox{ is expanding with factor } \lambda,\ \|f\|_{\CC^2} \leq \Gamma \right\},%
\end{equation*}
where ``\emph{expanding with factor $\lambda$}'' means that $|\rmD f(v)|\geq\lambda|v|$ holds for all $v\in TM$. We will consider an NDS $f_{\infty} = (f_n)_{n=0}^{\infty}$ on $M$ with $f_n \in \EC(\lambda,\Gamma)$ for fixed $\lambda,\Gamma$. It is clear that such a system is equicontinuous. We define%
\begin{equation*}
  \DC := \left\{\varphi:M\rightarrow\R\ :\ \varphi>0,\ \varphi \mbox{ is Lipschitz},\ \int \varphi \rmd m = 1 \right\}%
\end{equation*}
and for every $L>0$,%
\begin{equation*}
  \DC_L := \left\{\varphi \in \DC\ :\ \left|\frac{\varphi(x)}{\varphi(y)}-1\right| \leq Ld(x,y) \mbox{ if } d(x,y)<\ep\right\},%
\end{equation*}
where $\ep>0$ is a fixed number (depending on $\lambda$ and $\Gamma$), cf.~\cite[Sec.~2.2]{OSY}. Note that%
\begin{equation*}
  \DC = \bigcup_{L>0}\DC_L,%
\end{equation*}
since for every $\varphi\in\DC$ we have%
\begin{equation*}
  \left|\frac{\varphi(x)}{\varphi(y)}-1\right| = \frac{1}{\varphi(y)}\left|\varphi(x) - \varphi(y)\right| \leq \frac{\Lip(\varphi)}{\min \varphi}d(x,y).%
\end{equation*}
For any expanding $\CC^2$-map $f:M\rightarrow M$ we write%
\begin{equation*}
  \PC_f(\varphi)(x) = \sum_{y\in f^{-1}(x)}\frac{\varphi(y)}{|\det \rmD f(y)|},\qquad \PC_f(\varphi):M\rightarrow\R,%
\end{equation*}
for the Perron-Frobenius operator associated with $f$ acting on densities $\varphi\in\DC$. Note that this makes sense, since expanding maps are covering maps, and hence the sets $f^{-1}(x)$ are finite, all having the same number of elements. We have the following key result (cf.~\cite[Prop.~2.3]{OSY}).%

\begin{proposition}\label{prop_osy}
There exists $L^*>0$ for which the following holds. For any $L>0$ there is $\tau(L)\geq 1$ such that for all $\varphi\in\DC_L$ and $f_k\in\EC(\lambda,\Gamma)$ ($k\in\N_0$), $\PC_{f_0^n}(\varphi)\in\DC_{L^*}$ for all $n\geq\tau(L)$.%
\end{proposition}

\section{Metric and topological pressure}\label{sec_pressure}%

The topological pressure and its measure-theoretic counterpart for autonomous dynamical systems were introduced by Ruelle \cite{Ru1} and related (in full generality) by Walters \cite{Wal} via a variational principle. For nonautonomous systems, a notion of topological pressure was introduced by Huang, Wen and Zeng \cite{Hea}. We generalize this notion and also define a measure-theoretic counterpart.%

Given a topological NDS $(X_{\infty},f_{\infty})$ together with an IMS $\mu_{\infty} = (\mu_n)_{n\geq0}$ of Borel probability measures and a (uniformly) equicontinuous and uniformly bounded sequence $\varphi_{\infty}$ of functions $\varphi_n:X_n \rightarrow \R$, we define the metric pressure by%
\begin{eqnarray*}
  P_{\mu_{\infty}}(f_{\infty};\varphi_{\infty}) &:=& h_{\EC_{\Mis}}(f_{\infty};\mu_{\infty}) + \liminf_{n\rightarrow\infty}\int_{X_0}\frac{1}{n}\sum_{i=0}^{n-1} \varphi_i\circ f_0^i\rmd\mu_0\\
                                                &=& h_{\EC_{\Mis}}(f_{\infty};\mu_{\infty}) + \liminf_{n\rightarrow\infty}\frac{1}{n}\sum_{i=0}^{n-1}\int_{X_i} \varphi_i\rmd\mu_i.%
\end{eqnarray*}
In the following, we use the abbreviation $S_n\varphi_{\infty}(x) := \sum_{i=0}^{n-1}\varphi_i(f_0^i(x))$. We define the topological pressure by%
\begin{eqnarray*}
  S(n,\ep;\varphi_{\infty}) &:=& \sup\left\{ \sum_{x\in E} \rme^{S_n\varphi_{\infty}(x)}\ :\ E \subset X_0 \mbox{ is } (n,\ep)\mbox{-separated}\right\},\\
  R(n,\ep;\varphi_{\infty}) &:=& \inf\left\{ \sum_{x\in F} \rme^{S_n\varphi_{\infty}(x)}\ :\ F \subset X_0 \mbox{ is } (n,\ep)\mbox{-spanning}\right\},\\
  P_{\tp}(f_{\infty};\varphi_{\infty}) &:=& \lim_{\ep\searrow0}\limsup_{n\rightarrow\infty}\frac{1}{n}\log S(n,\ep;\varphi_{\infty})\\
                                        &=& \lim_{\ep\searrow0}\limsup_{n\rightarrow\infty}\frac{1}{n}\log R(n,\ep;\varphi_{\infty}).%
\end{eqnarray*}
Using the compactness of $X_0$, it is not hard to show that the supremum in the definition of $S(n,\ep;\varphi_{\infty})$ is in fact a maximum.%

\begin{remark}
The difference between our definition of $P_{\tp}$ and the one given in \cite{Hea} is that we consider a sequence of functions instead of a single function.%
\end{remark}

The proof of the following lemma is an adaptation of Walters \cite[Thm.~1.1]{Wal}.%

\begin{lemma}
The definition of $P_{\tp}(f_{\infty};\varphi_{\infty})$ is correct, i.e., the limits for $\ep\searrow0$ exist and the two expressions using $S(n,\ep;\varphi_{\infty})$ and $R(n,\ep;\varphi_{\infty})$, resp., coincide.%
\end{lemma}

\begin{proof}
The existence of the limits follows, because for $\ep_1<\ep_2$ every $(n,\ep_2)$-separated set is also $(n,\ep_1)$-separated, and hence $S(n,\ep_1;\varphi_{\infty}) \geq S(n,\ep_2;\varphi_{\infty})$. Similarly, every $(n,\ep_1)$-spanning set is also $(n,\ep_2)$-spanning, and hence $R(n,\ep_1;\varphi_{\infty}) \geq R(n,\ep_2;\varphi_{\infty})$. Now assume that $E\subset X_0$ is an $(n,\ep)$-separated set such that the sum $\sum_{x\in E}\rme^{S_n\varphi_{\infty}(x)}$ is maximal. Assume to the contrary the existence of $y \in X_0$ with $d_{0,n}(x,y) > \ep$ for all $x\in E$. Then also $E' := E \cup \{y\}$ is $(n,\ep)$-separated and the sum $\sum_{x\in E'}\rme^{S_n\varphi_{\infty}(x)}$ is strictly greater than $\sum_{x\in E}\rme^{S_n\varphi_{\infty}(x)}$, a contradiction. Hence, $E$ is a maximal $(n,\ep)$-separated set and thus also $(n,\ep)$-spanning, implying%
\begin{equation*}
  \lim_{\ep\searrow0}\limsup_{n\rightarrow\infty}\frac{1}{n}\log R(n,\ep;\varphi_{\infty}) \leq \lim_{\ep\searrow0}\limsup_{n\rightarrow\infty}\frac{1}{n}\log S(n,\ep;\varphi_{\infty}).%
\end{equation*}
To show the converse inequality, let $\delta>0$ and choose $\ep>0$ so that for all $n\geq1$,%
\begin{equation*}
  d_n(x,y) < \frac{\ep}{2} \qquad\Rightarrow\qquad |\varphi_n(x) - \varphi_n(y)| < \delta.%
\end{equation*}
Let $n\in\N$ and $\lambda>0$. Choose an $(n,\ep)$-separated set $E\subset X_0$ with%
\begin{equation*}
  S(n,\ep;\varphi_{\infty}) - \lambda < \sum_{x\in E}\rme^{S_n\varphi_{\infty}(x)},%
\end{equation*}
and choose an $(n,\ep/2)$-spanning set $F\subset X_0$ with%
\begin{equation*}
  \sum_{x\in F}\rme^{S_n\varphi_{\infty}(x)} - \lambda < R\left(n,\frac{\ep}{2};\varphi_{\infty}\right).%
\end{equation*}
Define a map $\sigma:E\rightarrow F$ by choosing for each $x\in E$ a point $\sigma(x)\in F$ with $d_{0,n}(x,\sigma(x))<\ep/2$. This map is injective, since $E$ is $(n,\ep)$-separated. Moreover,%
\begin{eqnarray*}
  \frac{\sum_{y\in F}\rme^{S_n\varphi_{\infty}(y)}}{\sum_{x\in E}\rme^{S_n\varphi_{\infty}(x)}} &\geq& \frac{\sum_{y\in\sigma(E)}\rme^{S_n\varphi_{\infty}(y)}}{\sum_{x\in E}\rme^{S_n\varphi_{\infty}(x)}}\\
  &\geq& \min_{x\in E} \exp\left(\sum_{i=0}^{n-1}\left[\varphi_i(f_0^i(\sigma(x))) - \varphi_i(f_0^i(x))\right]\right) \geq \rme^{-n\delta}.%
\end{eqnarray*}
Therefore,%
\begin{eqnarray*}
  R\left(n,\frac{\ep}{2};\varphi_{\infty}\right) &>& \sum_{y\in F}\rme^{S_n\varphi_{\infty}(y)} - \lambda\\
                                                &\geq& \rme^{-n\delta}\sum_{x\in E}\rme^{S_n\varphi_{\infty}(x)} - \lambda
                                                \geq \rme^{-n\delta}\left[S(n,\ep;\varphi_{\infty}) - \lambda\right] - \lambda.%
\end{eqnarray*}
Hence, $R(n,\ep/2;\varphi_{\infty})\geq \rme^{-n\delta} S(n,\ep;\varphi_{\infty})$, since $\lambda$ was chosen arbitrarily, implying%
\begin{equation*}
  \limsup_{n\rightarrow\infty}\frac{1}{n}\log R\left(n,\frac{\ep}{2};\varphi_{\infty}\right) \geq -\delta + \limsup_{n\rightarrow\infty}\frac{1}{n}\log S(n,\ep;\varphi_{\infty}).%
\end{equation*}
First sending $\ep$ and then $\delta$ to zero yields the desired inequality.%
\end{proof}

The topological pressure can also be defined in terms of sequences of open covers. For the sake of completeness, we also introduce this definition. Let $\UC_{\infty} = (\UC_n)_{n\geq0}$ be a sequence of open covers for $X_{\infty}$ and put%
\begin{equation*}
  T(n,\UC_{\infty};\varphi_{\infty}) := \inf\left\{\sum_{V\in\VC}\inf_{x\in V}\rme^{S_n\varphi_{\infty}(x)}:\ \VC \mbox{ is a finite subcover of } \bigvee_{i=0}^{n-1}f_0^{-i}\UC_i\right\}.%
\end{equation*}
We leave it to the reader to verify that
\begin{equation*}
  P_{\tp}(f_{\infty};\varphi_{\infty}) = \sup_{\UC_{\infty} \in \LC(X_{\infty})}\limsup_{n\rightarrow\infty}\frac{1}{n}\log T(n,\UC_{\infty};\varphi_{\infty}).%
\end{equation*}

In order to prove a variational inequality, relating the two notions of pressure, we need the following (partial) power rules.%

\begin{lemma}\label{lem_metricpr}
Assume that $f_{\infty}$ is equicontinuous and let $\varphi_{\infty}$ be equicontinuous and uniformly bounded. For each $k\geq1$ we define another sequence $\varphi^{[k]}_{\infty} = (\psi_n)_{n\geq0}$, $\psi_n:X_{nk}\rightarrow\R$, by%
\begin{equation*}
  \psi_n := \sum_{j=0}^{k-1}\varphi_{nk+j} \circ f_{nk}^j.% 
\end{equation*}
Then $\varphi^{[k]}_{\infty}$ is equicontinuous and uniformly bounded with%
\begin{equation}\label{eq_metricpowerrule}
  P_{\mu^{[k]}_{\infty}}\left(f^{[k]}_{\infty};\varphi^{[k]}_{\infty}\right) \geq k \cdot P_{\mu_{\infty}}\left(f_{\infty};\varphi_{\infty}\right),%
\end{equation}
where $\mu^{[k]}_{\infty}$ is the sequence defined by $\mu^{[k]}_n :\equiv \mu_{kn}$.%
\end{lemma}

\begin{proof}
By definition, the metric pressure of $f^{[k]}_{\infty}$ w.r.t.~$\varphi^{[k]}_{\infty}$ is given by%
\begin{eqnarray*}
  h_{\EC_{\Mis}}\left(f^{[k]}_{\infty};\mu^{[k]}_{\infty}\right) + \liminf_{n\rightarrow\infty}\frac{1}{n}\sum_{i=0}^{n-1}\int_{X_{ik}}\sum_{j=0}^{k-1}\varphi_{ik+j} \circ f_{ik}^j \rmd\mu_{ik}.%
\end{eqnarray*}
For the second term we obtain%
\begin{eqnarray*}
  \liminf_{n\rightarrow\infty}\frac{1}{n}\sum_{i=0}^{n-1}\sum_{j=0}^{k-1} \int_{X_{ik+j}} \varphi_{ik+j} \rmd\mu_{ik+j}
      &=& k\cdot \liminf_{n\rightarrow\infty}\frac{1}{nk}\sum_{l=0}^{nk-1} \int_{X_l} \varphi_l \rmd\mu_l\\
   &\geq& k \cdot \liminf_{n\rightarrow\infty}\frac{1}{n}\sum_{i=0}^{n-1} \int_{X_i} \varphi_i \rmd\mu_i.%
\end{eqnarray*}
Using the fact that the admissible class $\EC_{\Mis}(f_{\infty})^{[k]}$, given by ``restriction'' of the sequences $\PC_{\infty}\in\EC_{\Mis}(f_{\infty})$ to the spaces $X_{nk}$, $n\in\N_0$, is contained in $\EC_{\Mis}(f^{[k]}_{\infty})$, with the power rule for metric entropy (cf.~\cite[Prop.~3.9]{Kaw}) we find%
\begin{equation*}
  h_{\EC_{\Mis}(f_{\infty}^{[k]})}\left(f_{\infty}^{[k]};\mu_{\infty}^{[k]}\right) \geq k \cdot h_{\EC_{\Mis}(f_{\infty})}\left(f_{\infty};\mu_{\infty}\right),%
\end{equation*}
implying \eqref{eq_metricpowerrule}. Equicontinuity and boundedness of $\varphi^{[k]}_{\infty}$ are easy to see.%
\end{proof}

\begin{lemma}\label{lem_toppr}
For every $k\geq1$ it holds that%
\begin{equation*}
  P_{\tp}\left(f_{\infty}^{[k]};\varphi^{[k]}_{\infty}\right) \leq k \cdot P_{\tp}\left(f_{\infty};\varphi_{\infty}\right).%
\end{equation*}
\end{lemma}

\begin{proof}
Let $F\subset X_0$ be an $(nk,\ep;f_{\infty})$-spanning set. Then clearly $F$ is $(n,\ep;f_{\infty}^{[k]})$-spanning, implying%
\begin{equation*}
  R\left(n,\ep;\varphi^{[k]}_{\infty};f^{[k]}_{\infty}\right) \leq \sum_{x\in F}\rme^{S_n\varphi^{[k]}_{\infty}(x)} = \sum_{x\in F}\rme^{S_{nk}\varphi_{\infty}(x)}.
\end{equation*}
Since this holds for arbitrary $F$, we find $R(n,\ep;\varphi^{[k]}_{\infty};f^{[k]}_{\infty}) \leq R\left(nk,\ep;\varphi_{\infty};f_{\infty}\right)$, which yields the desired inequality of pressures.%
\end{proof}

Next we prove the variational inequality for pressure. We will use the following lemma, which can be found, e.g., in \cite[Lem.~20.2.2]{KHa}.%

\begin{lemma}\label{lem_hl}
If $\sum_{i=1}^n p_i = 1$, $p_i\geq0$, $a_i\in\R$, and $A = \sum_{i=1}^n \rme^{a_i}$, then $\sum_{i=1}^np_i(a_i-\log p_i)\leq \log A$ with equality if and only if $p_i = \rme^{a_i}/A$.%
\end{lemma}

\begin{theorem}\label{thm_varineq}
If $f_{\infty}$ is equicontinuous and $\varphi_{\infty}$ is equicontinuous and uniformly bounded, then%
\begin{equation*}
  P_{\mu_{\infty}}(f_{\infty};\varphi_{\infty}) \leq P_{\tp}(f_{\infty};\varphi_{\infty}).%
\end{equation*}
\end{theorem}

\begin{proof}
Take a sequence $\PC_{\infty}\in\EC_{\Mis}$. We may assume that each $\PC_n$ has the same number $k$ of elements, $\PC_n = \{P_{n,1},\ldots,P_{n,k}\}$. Exactly as in the proof of the variational principle for the entropy (cf.~\cite[Thm.~28]{Kaw}), we define $\QC_{\infty} = (\QC_n)_{n\geq0}$, $\QC_n = \{Q_{n,0},Q_{n,1},\ldots,Q_{n,k}\}$, with $H_{\mu_n}(\PC_n|\QC_n) \leq 1$ and obtain%
\begin{equation*}
  h(f_{\infty};\PC_{\infty}) \leq h(f_{\infty};\QC_{\infty}) + 1.%
\end{equation*}
By the choice of $\QC_{\infty}$ there is $\delta>0$ such that for all $n\geq0$,%
\begin{equation*}
  \min_{1\leq i < j \leq k} \min\left\{d_n(x,y)\ :\ (x,y) \in Q_{n,i}\tm Q_{n,j} \right\} \geq \delta.%
\end{equation*}
By equicontinuity of $\varphi_{\infty}$, there is $\alpha \in (0,\delta/2)$ such that $|\varphi_n(x)-\varphi_n(y)|\leq 1$ whenever $d_n(x,y) < \alpha$. For every $C \in \bigvee_{i=0}^{n-1}f_0^{-i}\QC_i$ there is $x_C \in \cl C$ with $S_n\varphi_{\infty}(x_C) = \sup_{x\in C} S_n\varphi_{\infty}(x)$. Now suppose that $E_n\subset X_0$ is $(n,\alpha)$-spanning. Then we can take $y_C \in E_n$ with $d_{0,n}(x_C,y_C)\leq \alpha$, and hence $S_n\varphi_{\infty}(x_C) \leq S_n\varphi_{\infty}(y_C)+n$. Note also that $\alpha < \delta/2$ implies $\#\{C\in\bigvee_{i=0}^{n-1}f_0^{-i}\QC_i\ : \ y_C=y\} \leq 2^n$ for all $y\in E_n$. (This is equivalent to the statement that every ball of radius $\alpha$ in $(X_i,d_i)$ intersects at most two elements of $\QC_i$, which holds by the definition of $\QC_{\infty}$.) Writing $\QC^{(n)} := \bigvee_{i=0}^{n-1}f_0^{-i}\QC_i$ and using Lemma \ref{lem_hl}, it follows that%
\begin{eqnarray*}
  H_{\mu_0}\left(\QC^{(n)}\right) + \int S_n\varphi_{\infty} \rmd\mu_0 &\leq& \sum_{C\in\QC^{(n)}}\mu_0(C)\left(-\log\mu_0(C) + S_n\varphi_{\infty}(x_C)\right)\allowdisplaybreaks\\
  &\leq& \log \sum_{C\in\QC^{(n)}}\rme^{S_n\varphi_{\infty}(y_C)+n}\allowdisplaybreaks\\
  &=& n + \log\left(\sum_{x\in E_n}\sum_{C\in\QC^{(n)}\atop y_C=x}\rme^{S_n\varphi_{\infty}(x)}\right)\allowdisplaybreaks\\
  &\leq& n + \log\left(2^n\sum_{x\in E_n}\rme^{S_n\varphi_{\infty}(x)}\right),%
\end{eqnarray*}
and thus%
\begin{equation*}
  \frac{1}{n}H_{\mu_0}(\QC^{(n)}) + \frac{1}{n}\int S_n\varphi_{\infty}\rmd\mu_0 \leq 1 + \log 2 + \frac{1}{n}\log\sum_{x\in E_n}\rme^{S_n\varphi_{\infty}(x)}.%
\end{equation*}
Hence, we obtain%
\begin{eqnarray*}
  h(f_{\infty};\PC_{\infty}) &+& \liminf_{n\rightarrow\infty}\frac{1}{n}\int S_n\varphi_{\infty}\rmd\mu_0\allowdisplaybreaks\\
   &\leq& 1 + h(f_{\infty};\QC_{\infty}) + \liminf_{n\rightarrow\infty}\frac{1}{n}\int S_n\varphi_{\infty}\rmd\mu_0\allowdisplaybreaks\\
    &\leq& 1 + \limsup_{n\rightarrow\infty}\left[\frac{1}{n}H_{\mu_0}(\QC^{(n)}) + \frac{1}{n}\int S_n \varphi_{\infty}\rmd\mu_0\right]\allowdisplaybreaks\\
    &\leq& 2 + \log 2 + \limsup_{n\rightarrow\infty}\frac{1}{n}\log\sum_{x\in E_n}\rme^{S_n\varphi_{\infty}(x)}.%
\end{eqnarray*}
Since $\alpha$ can be taken arbitrarily small and $\PC_{\infty}$ was chosen arbitrarily from $\EC_{\Mis}$,%
\begin{equation*}
  h_{\EC_{\Mis}}(f_{\infty}) + \liminf_{n\rightarrow\infty}\frac{1}{n}\int S_n\varphi_{\infty}\rmd\mu_0 \leq 2 + \log 2 + P_{\tp}(f_{\infty};\varphi_{\infty}).
\end{equation*}
The same estimate holds if we replace $f_{\infty}$ by $f_{\infty}^{[k]}$ and $\varphi_{\infty}$ by $\varphi_{\infty}^{[k]}$. Using the partial power rules, Lemma \ref{lem_metricpr} and Lemma \ref{lem_toppr}, this yields%
\begin{eqnarray*}
  P_{\mu_{\infty}}\left(f_{\infty};\varphi_{\infty}\right) &\leq& \frac{1}{k} P_{\mu_{\infty}^{[k]}}\left(f_{\infty}^{[k]};\varphi_{\infty}^{[k]}\right)\\
  &\leq& \frac{2 + \log 2 + P_{\tp}(f_{\infty}^{[k]};\varphi_{\infty}^{[k]})}{k} \leq \frac{2+\log 2}{k} + P_{\tp}(f_{\infty};\varphi_{\infty}).%
\end{eqnarray*}
Sending $k$ to infinity concludes the proof.%
\end{proof}

\section{Independence from the initial measure}\label{sec_indep}%

In this section, we prove that the metric entropy of an expanding NDS is independent of the initial measure as long as it has a positive Lipschitz density w.r.t.~the Riemannian volume. We recall the following property of the metric entropy from \cite[Prop.~9(vii)]{Kaw}.%

\begin{proposition}\label{prop_kkp1}
Let $(X_{\infty},f_{\infty})$ be a metric NDS and $\PC_{\infty}$ a sequence of finite measurable partitions. Then $h(f_{k,\infty};\PC_{k,\infty}) = h(f_{k+1,\infty};\PC_{k+1,\infty})$ for all $k\geq0$.%
\end{proposition}

The following result can be found as Example 36 in \cite{Kaw}.%

\begin{proposition}\label{prop_lmc}
Consider an NDS $(M,f_{\infty})$ with $f_n\in\EC(\lambda,\Gamma)$ and let $\varphi\in\DC$. Then the IMS defined by $\mu_0 := \varphi\rmd m$ and $\mu_n := f_0^n\mu_0$ for all $n\geq1$ has the property that the elements of the weak$^*$-closure of $\{\mu_n\}_{n\geq0}$ are pairwisely equivalent.%
\end{proposition}

The next corollary is an immediate consequence of the variational inequality \cite[Thm.~28]{Kaw} and \cite[Prop.~34]{Kaw}.%

\begin{corollary}
For any finite Borel partition $\PC$ of $M$ whose elements have boundaries of volume zero the inequality $h(f_{\infty};\PC;\mu_{\infty}) \leq h_{\tp}(f_{\infty})$ holds if $\mu_{\infty}$ is as in Proposition \ref{prop_lmc}.%
\end{corollary}

In order to prove our next result, we need to introduce the renormalization of densities in $\DC_{L^*}$, where $L^*$ is given by Proposition \ref{prop_osy}. For $\varphi \in \DC_{L^*}$ we put%
\begin{equation*}
  \hat{\varphi} := \frac{\varphi-\frac{1}{2}\kappa}{1-\frac{1}{2}\kappa \cdot m(M)},%
\end{equation*}
where $\kappa>0$ is a fixed lower bound for the functions in $\DC_{L^*}$, whose existence is guaranteed by the proof of Proposition \ref{prop_lmc} (see \cite[Ex.~36]{Kaw}). In \cite{OSY} we find the following lemma.%

\begin{lemma}\label{lem_renorm}
For every $\varphi\in\DC_{L^*}$ it holds that $\hat{\varphi}\in\DC_{2L^*}$.%
\end{lemma}

\begin{lemma}\label{lem_suplem}
Assume $f_i\in\EC(\lambda,\Gamma)$ for $i\geq0$. Then for every $N\in\N$ there exists a constant $C>0$ such that
\begin{equation*}
  \sup_{x\in M}\left|(\PC_{f_i^r}\varphi)(x)\right| \leq C \sup_{x\in M}|\varphi(x)|%
\end{equation*}
for all $\varphi\in\CC^0(M)$, $i\geq0$ and $r \in \{0,1,\ldots,N\}$.%
\end{lemma}

\begin{proof}
We may assume $r\geq1$. Since the expansion factor $\lambda$ is fixed for all $f_i$,%
\begin{equation*}
  \left|\det \rmD f_i^r(x)\right| = \prod_{l=0}^{r-1} \left|\det \rmD f_{i+l}(f_i^l(x))\right| \geq \prod_{l=0}^{r-1} \lambda^{\dim M} \geq \lambda^{\dim M}.%
\end{equation*}
Since the $\CC^1$-norm of all $f_i$ is bounded by $\Gamma$, $\#f_i^{-r}(x)$ is bounded. To see this, take $f\in\EC(\lambda,\Gamma)$ and consider an open cover of $M$ consisting of finitely many evenly covered connected sets $V_1,\ldots,V_n$. By taking appropriate intersections, we may assume that the $V_i$ form a partition of $M$ (though they will no longer be open, but still measurable). The preimage $f^{-1}(V_i)$ has the same number of components for each $i$, say $f^{-1}(V_i) = U_{i1}\cup\cdots\cup U_{ik}$, where the $U_{ij}$ are pairwisely disjoint. Then the volumes of the sets $V_i$ sum up to $m(M)=1$, and the same is true for the volumes of the sets $U_{ij}$, $1\leq i\leq n$, $1\leq j\leq k$. Moreover,%
\begin{equation*}
  m(V_i) = m(f(U_{ij})) \leq \Gamma^{\dim M} m(U_{ij}).%
\end{equation*}
Altogether, we obtain the estimate%
\begin{equation*}
  1 = \sum_{i,j}m(U_{ij}) \geq \sum_{i,j}\frac{1}{\Gamma^{\dim M}} m(V_i) = \frac{k}{\Gamma^{\dim M}}.%
\end{equation*}
Hence, we see that $k = \# f^{-1}(x) \leq \Gamma^{\dim M}$. Since each $f_i^r$ is the composition of at most $N$ elements of $\EC(\lambda,\Gamma)$, we have $\# f_i^{-r}(x) \leq \Gamma^{N \dim M}$. The assertion of the lemma now follows from the estimate%
\begin{equation*}
  \left|(\PC_{f_i^r}\varphi)(x)\right| \leq \sum_{y \in f_i^{-r}(x)}\frac{|\varphi(y)|}{|\det \rmD f_i^r(y)|} \leq \frac{\Gamma^{N \dim M}}{\lambda^{\dim M}} \sup_{y\in M} |\varphi(y)|.%
\end{equation*}
\end{proof}

\begin{theorem}
Consider an NDS $(M,f_{\infty})$ with $f_n \in \EC(\lambda,\Gamma)$. Then for any two initial densities $\varphi,\psi\in\DC$ and any sequence $\PC_{\infty}$ of Borel partitions of $M$ with uniformly bounded number of elements such that the volumes of the elements of each $\PC_n$ are sufficiently small (uniformly in $n$) it holds that%
\begin{equation*}
  h(f_{\infty};\PC_{\infty};\mu_{\infty}) = h(f_{\infty};\PC_{\infty};\nu_{\infty}),%
\end{equation*}
where $\mu_n = f_0^n(\varphi\rmd m)$ and $\nu_n = f_0^n(\psi\rmd m)$. Consequently,%
\begin{equation*}
  h_{\EC_M}(f_{\infty};\mu_{\infty}) = h_{\EC_M}(f_{\infty};\nu_{\infty}).%
\end{equation*}
\end{theorem}

\begin{proof}
We prove the theorem in three steps.%

\emph{Step 1.} Let $\varphi_n = \PC_{f_0^n}(\varphi)$ and $\psi_n = \PC_{f_0^n}(\psi)$ for all $n\geq0$. We prove that for every Borel set $A\subset M$ we have an exponential estimate%
\begin{equation}\label{eq_expest}
  \int_A |\varphi_n - \psi_n| \rmd m \leq m(A) \cdot K \mu^n \mbox{\qquad for all\ } n\geq0%
\end{equation}
with constants $K>0$ and $\mu \in (0,1)$, which are independent of $A$. Except for one point, the proof is the same as the one for $A=M$ (cf.~\cite[Thm.~1]{OSY}). In view of Proposition \ref{prop_osy}, we may assume $\varphi,\psi \in \DC_{L^*}$ for some $L^*>0$. There exists a uniform lower bound $\kappa>0$ for all functions in $\DC_{L^*}$ (cf.~the proof of Proposition \ref{prop_lmc} in \cite[Ex.~36]{Kaw}). Let $N = \tau(2L^*)$ be given by Proposition \ref{prop_osy}, and put%
\begin{equation}\label{eq_renorm}
  \hat{\varphi} := \frac{\varphi - \frac{1}{2}\kappa}{1-\frac{1}{2}\kappa\cdot m(M)},\qquad \hat{\psi} := \frac{\psi - \frac{1}{2}\kappa}{1-\frac{1}{2}\kappa\cdot m(M)}.%
\end{equation}
Then $\bar{\varphi}_N := \PC_{f_0^N}(\hat{\varphi})$ and $\bar{\psi}_N := \PC_{f_0^N}(\hat{\psi})$ are in $\DC_{L^*}$. We subtract $\frac{1}{2}\kappa$ from each of $\bar{\varphi}_N$ and $\bar{\psi}_N$ and renormalize as in \eqref{eq_renorm}, obtaining $\hat{\varphi}_N$ and $\hat{\psi}_N$, respectively. By Lemma \ref{lem_renorm}, they are in $\DC_{2L^*}$. In general, given $\hat{\varphi}_{(k-1)N},\hat{\psi}_{(k-1)N}\in\DC_{2L^*}$, we let%
\begin{equation*}
  \bar{\varphi}_{kN} := \PC_{f_{(k-1)N}^N}(\hat{\varphi}_{(k-1)N}),\qquad \bar{\psi}_{kN} := \PC_{f_{(k-1)N}^N}(\hat{\psi}_{(k-1)N}).%
\end{equation*}
By Proposition \ref{prop_osy}, $\bar{\varphi}_{kN},\bar{\psi}_{kN}\in\DC_{L^*}$. We subtract $\frac{1}{2}\kappa$ and renormalize to obtain $\hat{\varphi}_{kN},\hat{\psi}_{kN}$ in $\DC_{2L^*}$, completing the induction.%

By this process, for $kN \leq n < (k+1)N$ and Lemma \ref{lem_suplem} we obtain%
\begin{eqnarray*}
&& \int_A \left|\varphi_n - \psi_n\right| \rmd m = \int_A \left|\PC_{f_0^n}(\varphi) - \PC_{f_0^n}(\psi)\right|\rmd m\allowdisplaybreaks\\
  && \leq \left(1 - \frac{1}{2}\kappa\cdot m(M)\right)^k \int_A \left|\PC_{f_{kN}^{n-kN}}(\bar{\varphi}_{kN}) - \PC_{f_{kN}^{n-kN}}(\bar{\psi}_{kN})\right|\rmd m\allowdisplaybreaks\\
  && \leq \left(1 - \frac{1}{2}\kappa\cdot m(M)\right)^k m(A) \sup_{x\in M} \left|\PC_{f_{kN}^{n-kN}}\left(\bar{\varphi}_{kN}-\bar{\psi}_{kN}\right)(x)\right|\allowdisplaybreaks\\
  && \leq \left(1 - \frac{1}{2}\kappa\cdot m(M)\right)^k m(A) C \cdot \sup_{x\in M} \left|\left(\bar{\varphi}_{kN}-\bar{\psi}_{kN}\right)(x)\right|\allowdisplaybreaks\\
  && \leq \left(1 - \frac{1}{2}\kappa\cdot m(M)\right)^k m(A) C \cdot \left(\sup_{x\in M} \left|\bar{\varphi}_{kN}(x)\right| + \sup_{x\in M} \left|\bar{\psi}_{kN}(x)\right|\right)\allowdisplaybreaks\\
  && \leq \left(1 - \frac{1}{2}\kappa\cdot m(M)\right)^k m(A) C \cdot 2 \sup_{\alpha\in\DC_{L^*}} \|\alpha\|_{\CC^0}.%
\end{eqnarray*}
Since the functions in $\DC_{L^*}$ are uniformly bounded, $\sup_{\alpha\in\DC_{L^*}}\|\alpha\|_{\CC^0} < \infty$. This easily implies the desired estimate \eqref{eq_expest}.%

\emph{Step 2.} Let $\AC$ be a finite measurable partition of $M$ and $\eta(x) = x\log(x)$, defined on $[0,1]$ with $\eta(0)=0$. Then, using the mean value theorem, we get%
\begin{eqnarray*}
  && \left|\sum_{A\in\AC}\mu_n(A)\log\mu_n(A) - \sum_{A\in\AC}\nu_n(A)\log\nu_n(A)\right| \leq \sum_{A\in\AC}\left|\eta(\mu_n(A))-\eta(\nu_n(A))\right|\\
&& = \sum_{A\in\AC}|1 + \log(\xi_A)||\mu_n(A)-\nu_n(A)|
 \stackrel{\eqref{eq_expest}}{\leq} K\mu^n \sum_{A\in\AC}|1 + \log(\xi_A)| m(A)%
\end{eqnarray*}
for some $\xi_A$ between $\mu_n(A)$ and $\nu_n(A)$. In the case $\mu_n(A) = \nu_n(A) = 0$, the corresponding summand is zero by convention. If the volumes of the sets in $\AC$ are sufficiently small, we have $|1 + \log(\xi_A)| = -1-\log(\xi_A)$. We also have $\kappa_* \leq \varphi_n,\psi_n$ for a constant $\kappa_* > 0$ and all sufficiently large $n$ (cf.~the proof of Proposition \ref{prop_lmc} in \cite[Ex.~36]{Kaw}). Then $\mu_n(A),\nu_n(A) \geq m(A)\kappa_*$, and hence $\xi_A \geq m(A)\kappa_*$, implying $-\log(\xi_A) \leq -\log(\kappa_* m(A))$. We thus obtain%
\begin{eqnarray*}
&&  \left|\sum_{A\in\AC}\mu_n(A)\log\mu_n(A) - \sum_{A\in\AC}\nu_n(A)\log\nu_n(A)\right|\allowdisplaybreaks\\
&& \leq -K\mu^n\sum_{A\in\AC}(1+\log(\kappa_*) + \log m(A)) m(A)\allowdisplaybreaks\\
&& = -K\mu^n \left[1+\log(\kappa_*) + \sum_{A\in\AC}m(A)\log m(A)\right] = c_n + K\mu^nH_m(\AC)%
\end{eqnarray*}
with $c_n = -K\mu^n(1+\log(\kappa_*))$.%

\emph{Step 3.} Now let $\PC_{\infty}$ be a sequence of partitions of $M$ such that the volumes of the sets in each $\PC_n$ are smaller than $\rme^{-1}$ so that the conclusions of Step 2 hold. Using Proposition \ref{prop_kkp1} and the general estimate $|\limsup_t a_t  - \limsup_t b_t| \leq \limsup_t |a_t - b_t|$, we obtain%
\begin{eqnarray*}
  && \left| h(f_{0,\infty};\PC_{0,\infty};\mu_{0,\infty}) - h(f_{0,\infty};\PC_{0,\infty};\nu_{0,\infty}) \right|\\
  && = \left| h(f_{n,\infty};\PC_{n,\infty};\mu_{n,\infty}) - h(f_{n,\infty};\PC_{n,\infty};\nu_{n,\infty}) \right|\\
  && \leq \limsup_{k\rightarrow\infty}\frac{1}{k} \left|H_{\mu_n}\left( \bigvee_{i=0}^{k-1} f_n^{-i}\PC_{n+i} \right) - H_{\nu_n}\left( \bigvee_{i=0}^{k-1} f_n^{-i}\PC_{n+i} \right)\right|\\
  && \stackrel{\mathrm{Step\ 2}}{\leq} \limsup_{k\rightarrow\infty}\frac{1}{k}\left(c_n + K\mu^nH_m\left(\bigvee_{i=0}^{k-1} f_n^{-i}\PC_{n+i}\right)\right)\\
  && = K\mu^nh\left(f_{n,\infty};\PC_{n,\infty};m_{\infty}\right)\\
  && \leq K\mu^n \log \sup_{i\geq0}\#\PC_i,%
\end{eqnarray*}
where $m_{\infty}$ denotes the sequence $m,f_0^1 m,f_0^2 m,\ldots$. Sending $n$ to infinity yields the result.%
\end{proof}

\section{Entropy formulas}\label{sec_entforms}%

In this section, we derive formulas for the metric entropy of an expanding NDS w.r.t.~the invariant sequence $m_{\infty} = (m,f_0^1m,f_0^2m,\ldots)$ and for the topological entropy. The key for the proof is the following distortion lemma (\cite[Lem.~2.6]{OSY}).%

\begin{lemma}\label{lem_dist}
For every sufficiently small $\ep$ there exists $C_0>0$ such that%
\begin{equation*}
  \frac{|\det \rmD f_0^n(x)|}{|\det \rmD f_0^n(y)|} \leq \rme^{C_0 d(f_0^n(x),f_0^n(y))},%
\end{equation*}
if $x,y\in M$ and $n\geq0$ such that $d(f_0^k(x),f_0^k(y)) < \ep$ for $k=0,1,\ldots,n-1$.%
\end{lemma}

\begin{lemma}\label{lem2}
There exists $\rho>0$ such that every $\rho$-ball in $M$ is evenly covered by each $f \in \EC(\lambda,\Gamma)$ and every branch of the inverse map is a contraction.%
\end{lemma}

\begin{proof}
We claim that there exists $\delta>0$ such that each $f$ is a diffeomorphism on every ball of radius $\delta$. It is clear that each $f$ is locally injective (by the expansion property and the inverse function theorem). The existence of a uniform radius of injectivity easily follows from the proof of the inverse function theorem, the uniform expansion constant and the uniform bound $\Gamma$ on the second derivative of $f$ (cf., for instance, \cite[Thm.~1.2]{Lan}). Now consider any $f\in\EC(\lambda,\Gamma)$, $y\in M$ and $x\in f^{-1}(y)$. If $z \in B_{\delta}(y)$, then there exists a shortest geodesic $\gamma:[0,1]\rightarrow M$ from $y$ to $z$ with $\gamma([0,1]) \subset B_{\delta}(y)$. Since $f$ is a covering map, there exists a lift $\tilde{\gamma}:[0,1]\rightarrow M$ with $\tilde{\gamma}(0) = x$, i.e., $f \circ \tilde{\gamma} = \gamma$. We find%
\begin{equation*}
  \delta > d(y,z) = \LC(\gamma) = \int_0^1 |\rmD f(\tilde{\gamma}(s))\dot{\tilde{\gamma}}(s)|\rmd s \geq \lambda \int_0^1 |\dot{\tilde{\gamma}}(s)|\rmd s = \lambda \LC(\tilde{\gamma}),%
\end{equation*}
implying $d(x,\tilde{\gamma}(1)) \leq \LC(\tilde{\gamma}) < \delta/\lambda < \delta$. In particular, $z = \gamma(1) = f(\tilde{\gamma}(1)) \in f(B_{\delta}(x))$. This implies%
\begin{equation*}
  B_{\delta}(f(x)) \subset f(B_{\delta}(x)) \mbox{\quad for all\ } f\in\EC(\lambda,\Gamma),\ x \in M.%
\end{equation*}
Of course, the same inclusion holds for any radius smaller than $\delta$. Let $\rho := \delta/2$, $y\in M$, $f\in\EC(\lambda,\Gamma)$ and write $f^{-1}(y) = \{x_1,\ldots,x_k\}$. For $i\neq j$, $x_i$ and $x_j$ have at least distance $\delta$ to each other, and hence $B_{\rho}(x_i) \cap B_{\rho}(x_j) = \emptyset$. Put $W := \bigcap_{i=1}^k f(B_{\rho}(x_i))$. Then $W$ is an open neighborhood of $y$ that contains $B_{\rho}(y)$. This easily implies that $B_{\rho}(y)$ is evenly covered. It is clear that the branches of the inverse are contractions.%
\end{proof}

To compute the entropy, we need a simple version of the Bowen-Ruelle volume lemma.%

\begin{proposition}\label{prop_volumelemma}
Let $(M,f_{\infty})$ be an NDS with $f_n \in \EC(\lambda,\Gamma)$. Then there exist $\ep>0$ and $C = C(\ep)>0$, $D = D(\ep)>0$ such that for all $n\geq0$ and $x\in M$,%
\begin{equation*}
  D \left|\det \rmD f_0^n(x)\right|^{-1} \leq m(B^n_0(x,\ep)) \leq C \left|\det \rmD f_0^n(x)\right|^{-1}.%
\end{equation*}
\end{proposition}

\begin{proof}
For the proof of the first inequality we use that for small $\ep$ it holds that%
\begin{equation}\label{eq_vh}
  (f_0^n)_x^{-1}B_{\ep}(f_0^n(x)) \subset B^n_0(x,\ep),%
\end{equation}
where $(f_0^n)_x$ denotes the restriction of $f_0^n$ to a small neighborhood of $x$ on which $f_0^n$ is injective. In fact, it follows from Lemma \ref{lem2} that there is $\ep$ small enough so that every $\ep$-ball is evenly covered by all $f_0^n$, $n\geq1$. For such $\ep$, $(f_0^n)_x^{-1}$ is defined on $B_{\ep}(f_0^n(x))$. The map $(f_0^n)_x^{-1}$ can be decomposed as%
\begin{equation*}
  (f_0^n)_x^{-1} = (f_0)_x^{-1} \circ \cdots \circ (f_{n-1})_{f_0^{n-1}(x)}^{-1}.%
\end{equation*}
Assume $y \in (f_0^n)_x^{-1} B_{\ep}(f_0^n(x))$, $y = (f_0^n)_x^{-1}(z)$, and let $i \in \{0,1,\ldots,n\}$. Then%
\begin{eqnarray*}
  d\left(f_0^i(y),f_0^i(x)\right) &=& d\left(f_0^i( (f_0)_x^{-1} \circ \cdots \circ (f_{n-1})_{f_0^{n-1}(x)}^{-1}(z)),f_0^i(x)\right)\\
                                  &=& d\left((f_i)_{f_0^i(x)}^{-1} \circ \cdots \circ (f_{n-1})_{f_0^{n-1}(x)}^{-1}(z),f_0^i(x)\right).%
\end{eqnarray*}
We can write%
\begin{equation*}
  f_0^i(x) = (f_i)_{f_0^i(x)}^{-1} \circ \cdots \circ (f_{n-1})_{f_0^{n-1}(x)}^{-1}(f_0^n(x)).%
\end{equation*}
Since each $(f_j)_{f_0^j(x)}^{-1}$ is a contraction, this yields%
\begin{equation*}
  d\left(f_0^i(y),f_0^i(x)\right) \leq d(z,f_0^n(x)) < \ep \mbox{\quad for\ } i=0,1,\ldots,n.%
\end{equation*}
Hence, we have $y \in B^n_0(x,\ep)$, implying \eqref{eq_vh}. We thus obtain%
\begin{eqnarray*}
  m(B^n_0(x,\ep)) &\geq& m\left((f_0^n)_x^{-1}B_{\ep}(f_0^n(x))\right) = \int_{B_{\ep}(f_0^n(x))} \left|\det \rmD (f_0^n)_x^{-1}(u)\right| \rmd m(u)\\
  &=& \int_{B_{\ep}(f_0^n(x))} \left|\det \rmD f_0^n( (f_0^n)_x^{-1}(u))\right|^{-1} \rmd m(u).%
\end{eqnarray*}
By Lemma \ref{lem_dist}, we have%
\begin{equation*}
  \left|\det \rmD f_0^n((f_0^n)_x^{-1}(y))\right| \leq \rme^{C_0 d(y,f_0^n(x))} \left|\det \rmD f_0^n(x)\right|,%
\end{equation*}
implying%
\begin{eqnarray*}
   m(B^n_0(x,\ep)) &\geq& \int_{B_{\ep}(f_0^n(x))} \rme^{-C_0 d(y,f_0^n(x))}\rmd m(y) \cdot \left|\det \rmD f_0^n(x)\right|^{-1}\\
                   &\geq& m\left(B_{\ep}(f_0^n(x))\right) \cdot \rme^{-C_0 \ep} \left|\det \rmD f_0^n(x)\right|^{-1}\\
                   &\geq& \min_{z\in M} m(B_{\ep}(z)) \cdot \rme^{-C_0\ep} \left|\det \rmD f_0^n(x)\right|^{-1}.%
\end{eqnarray*}
This concludes the proof of the first inequality.%

To prove the converse inequality, note that there is $\delta>0$ such that two different preimages of a point $x\in M$ are at least $\delta$ apart under each of the maps $f_n$, because the maps have a common radius of injectivity as shown in the proof of Lemma \ref{lem2}. Now choose $\ep \in (0,\delta/4)$ such that every $\ep$-ball in $M$ is evenly covered by each of the maps $f_n$. Then any Bowen-ball $B^n_0(x,\ep)$ is contained in precisely one leaf over $B_{\ep}(f_0^n(x))$ under the covering map $f_0^n$. Indeed, if $y_1,y_2\in B^n_0(x,\ep)$, then $d(y_1,y_2) \leq d(y_1,x) + d(x,y_2) < \delta/2$. The leaves over $B_{\ep}(f_0(x))$ are contained in the $\delta/4$-balls around the preimages of $f_0(x)$. If $f_0(\tilde{x}) = f_0(x)$ and $y_1 \in B_{\delta/4}(x)$, $y_2 \in B_{\delta/4}(\tilde{x})$, then $d(x,\tilde{x}) \leq d(x,y_1) + d(y_1,y_2) + d(y_2,\tilde{x}) < \delta/4 + \delta/2 + \delta/4 = \delta$, implying $x = \tilde{x}$. Inductively, we find that $f_0^i(y_1)$ and $f_0^i(y_2)$ are in the same leaf over $B_{\ep}(f_0^{i+1}(x))$ for $0\leq i\leq n-1$, and hence%
\begin{equation*}
  B^n_0(x,\ep) \subset (f_0)_{x}^{-1} \circ \cdots \circ (f_{n-1})_{f_0^{n-1}(x)}^{-1}(B_{\ep}(f_0^n(x))).%
\end{equation*}
Now the desired inequality follows similarly as the first one, using Lemma \ref{lem_dist} again.%
\end{proof}

\begin{theorem}\label{thm_me_form}
For an NDS $(M,f_{\infty})$ with $f_n\in\EC(\lambda,\Gamma)$ it holds that%
\begin{equation}\label{eq_entform}
  h_{\EC_{\Mis}}(f_{\infty};m_{\infty}) = \limsup_{n\rightarrow\infty}\frac{1}{n}\int_M \log \left|\det \rmD f_0^n(x)\right| \rmd m(x).%
\end{equation}
\end{theorem}

\begin{proof}
We prove the theorem in two steps.%

\emph{Step 1.} We prove the inequality ``$\leq$'' in \eqref{eq_entform}. By Proposition \ref{prop_volumelemma}, for all sufficiently small $\ep$ there are constants $C_{\ep},D_{\ep}>0$ with%
\begin{equation}\label{eq_vollemma}
  D_{\ep} \leq m\left(B^n_0(x,\ep)\right)\left|\det \rmD f_0^n(x)\right| \leq C_{\ep}%
\end{equation}
for all $n\geq0$ and $x\in M$. Putting $\varphi_n(x) := -\log|\det\rmD f_n(x)|$, we can show that (i) $\varphi_{\infty} = (\varphi_n)_{n\geq0}$ is equicontinuous and uniformly bounded and (ii) $P_{\tp}(f_{\infty};\varphi_{\infty})=0$. By the variational inequality (Theorem \ref{thm_varineq}) this implies%
\begin{equation*}
  h_{\EC_{\Mis}}(f_{\infty};m_{\infty}) + \liminf_{n\rightarrow\infty}\frac{1}{n}\int_M S_n\varphi_{\infty} \rmd m \leq 0,%
\end{equation*}
which yields%
\begin{eqnarray*}
  h_{\EC_{\Mis}}(f_{\infty};m_{\infty}) &\leq& \limsup_{n\rightarrow\infty}\frac{1}{n}\int_M \sum_{i=0}^{n-1} \left( -\varphi_i \circ f_0^i \right) \rmd m \\
                                           &=& \limsup_{n\rightarrow\infty}\frac{1}{n}\int_M \log \left|\det \rmD f_0^n(x)\right|\rmd m(x).%
\end{eqnarray*}
Equicontinuity and boundedness of $\varphi_{\infty}$ are clear, since each $\varphi_n$ is a $\CC^1$-function and these functions together with their derivatives are uniformly bounded. The proof for $P_{\tp}(f_{\infty};\varphi_{\infty})=0$ follows from \eqref{eq_vollemma}: Let $E\subset M$ be an $(n,\ep)$-separated set for a small $\ep$. Then the balls $B^n_0(x,\ep/2)$, $x\in E$, are disjoint, and hence%
\begin{eqnarray*}
  \sum_{x\in E}\rme^{S_n\varphi_{\infty}(x)} &=& \sum_{x\in E}|\det\rmD f_0^n(x)|^{-1}\\
  &\leq& \frac{1}{D_{\ep/2}}\sum_{x\in E}m\left(B^n_0(x,\ep/2)\right) \leq \frac{1}{D_{\ep/2}}m(M),%
\end{eqnarray*}
implying $P_{\tp}(f_{\infty};\varphi_{\infty}) \leq 0$. Using the other half of the volume lemma, analogously we find that $\sum_{x\in F}\rme^{S_n\varphi_{\infty}(x)} \geq C_{\ep}^{-1}m(M)$ for any $(n,\ep)$-spanning set $F$, and hence $P_{\tp}(f_{\infty};\varphi_{\infty}) \geq 0$.%

\emph{Step 2.} We prove the converse inequality. To this end, we use the notation $\PC^{(n)} := \bigvee_{i=0}^nf_0^{-i}\PC_i$ and write%
\begin{equation*}
  I_{\PC^{(n)}}:M \rightarrow \R,\qquad I_{\PC^{(n)}}(x) = -\log m(P_x),%
\end{equation*}
for the associated information function. Here $P_x$ is the unique element of $\PC^{(n)}$ such that $x \in P_x$. Using this notation, we obtain%
\begin{equation}\label{eq_informationfctint}
  H_m(\PC^{(n)}) = \int_M I_{\PC^{(n)}}(x) \rmd m(x).%
\end{equation}
Now let us assume that the diameter of each element of each partition $\PC_n$ is smaller than a given $\ep>0$. Then every element of the partition $\PC^{(n)}$ is contained in the Bowen-ball $B^n_0(x,\ep)$ around any of its elements $x$. Using \eqref{eq_informationfctint}, this implies%
\begin{equation}\label{eq_bowenballint}
  H_m(\PC^{(n)}) \geq \int_M -\log m(B_0^n(x,\ep)) \rmd m(x).%
\end{equation}
From \eqref{eq_vollemma} and \eqref{eq_bowenballint} we obtain for sufficiently small $\ep>0$ that%
\begin{eqnarray*}
  \limsup_{n\rightarrow\infty}\frac{1}{n}H_m(\PC^{(n)}) &\geq& \limsup_{n\rightarrow\infty}\frac{1}{n}\int_M -\log\left( C_{\ep} \cdot \left|\det \rmD f_0^n(x)\right|^{-1} \right) \rmd m(x)\\
                                                           &=& \limsup_{n\rightarrow\infty}\frac{1}{n}\int_M \log\left|\det \rmD f_0^n(x)\right| \rmd m(x).%
\end{eqnarray*}
From Proposition \ref{prop_lmc} and \cite[Prop.~4.4]{Kaw} it follows that $\EC_{\Mis}$ contains all constant sequences of partitions whose members have boundaries of volume zero. Since there exist such partitions with arbitrarily small diameters (cf.~\cite[Lem.~4.5.1]{KHa}), we are done.% 
\end{proof}

Using again the volume lemma, we can also provide a formula for the topological entropy of an expanding NDS.%

\begin{theorem}
For an NDS $(M,f_{\infty})$ with $f_n\in\EC(\lambda,\Gamma)$ it holds that%
\begin{equation*}
  h_{\tp}(f_{\infty}) = \limsup_{n\rightarrow\infty}\frac{1}{n}\log\int_M |\det \rmD f_0^n(x)|\rmd m(x).%
\end{equation*}
\end{theorem}

\begin{proof}
The proof is divided into two steps.%

\emph{Step 1.} First we prove the inequality ``$\leq$''. If $E\subset M$ is a maximal $(n-1,\ep)$-separated set, then the balls $B^{n-1}_0(x,\ep/2)$, $x\in E$, are disjoint. Hence, we find%
\begin{equation}\label{eq_te_firstest}
  \int_M |\det \rmD f_0^n(x)| \rmd m(x) \geq \sum_{x\in E} \int_{B^{n-1}_0(x,\ep/2)} |\det \rmD f_0^n(y)| \rmd m(y).%
\end{equation}
By Lemma \ref{lem_dist}, there exists a constant $C_0$ such that%
\begin{equation*}
  \left|\det \rmD f_0^n(x)\right| \leq \rme^{C_0 d(f_0^n(x),f_0^n(y))}\left|\det \rmD f_0^n(y)\right|%
\end{equation*}
for all $y \in B^{n-1}_0(x,\ep)$, $x\in M$. Since the $f_n$ are Lipschitz continuous with a common Lipschitz constant $C$, we have $d(f_0^n(x),f_0^n(y)) \leq C\ep$, and hence%
\begin{equation*}
  \left|\det \rmD f_0^n(y)\right| \geq \rme^{-C\ep}\left|\det \rmD f_0^n(x)\right|,%
\end{equation*}
which, together with the volume lemma, implies%
\begin{eqnarray*}
  && \int_{B^{n-1}_0(x,\ep/2)}|\det\rmD f_0^n(x)| \rmd m(x) \geq \rme^{-C\ep} \left|\det \rmD f_0^n(x)\right| m\left(B^{n-1}_0\left(x,\frac{\ep}{2}\right)\right)\\
	 && \qquad\qquad \geq \rme^{-C\ep}D_{\ep/2}|\det\rmD f_{n-1}(x)| \geq \rme^{-C\ep}D_{\ep/2}\lambda^{\dim M} =: b_{\ep}.%
\end{eqnarray*}
Together with \eqref{eq_te_firstest} this gives%
\begin{equation*}
  \int_M |\det \rmD f_0^n(x)| \rmd m(x) \geq b_{\ep} \cdot r_{\sep}(n-1,\ep;f_{\infty}).%
\end{equation*}
This yields the desired estimate.%

\emph{Step 2.} The proof for the converse inequality is similar. Here we let $F\subset M$ be a minimal $(n-1,\ep)$-spanning set and obtain%
\begin{equation*}
  \int_M |\det \rmD f_0^n(x)| \rmd m(x) \leq \sum_{x\in F} \int_{B^{n-1}_0(x,\ep)} |\det \rmD f_0^n(y)| \rmd m(y).%
\end{equation*}
Using the distortion lemma and the volume lemma again, we find%
\begin{equation*}
  \int_{B^{n-1}_0(x,\ep)} |\det \rmD f_0^n(y)| \rmd m(y) \leq c_{\ep}%
\end{equation*}
for a constant $c_{\ep}>0$, and hence%
\begin{equation*}
  \int_M |\det \rmD f_0^n(x)| \rmd m(x) \leq c_{\ep} r_{\spn}(n-1,\ep;f_{\infty}),%
\end{equation*}
implying the lower estimate.%
\end{proof}

\begin{remark}
Note that Jensen's inequality gives%
\begin{equation*}
  \int_M \log|\det \rmD f_0^n(x)|\rmd m(x) \leq \log \int_M |\det \rmD f_0^n(x)|\rmd m(x)%
\end{equation*}
for all $n$, showing the inequality between metric and topological entropy which we already know from the variational inequality.%
\end{remark}

It is clear that the expressions for the metric and the topological entropy in general do not coincide. (They do coincide, e.g., if the $f_n$ are algebraic torus endomorphisms.) In fact, this is already so in the autonomous case, where it is well-known that the absolutely continuous invariant measure of a $\CC^2$-expanding map $f$ is not necessarily a measure of maximal entropy (cf.~Walters \cite{Wa2}). However, this measure is an equilibrium state for the pressure w.r.t.~the potential $\varphi(x) = -\log |\det \rmD f(x)|$. From our results we find that the analogous statement is true in the nonautonomous case.%

\begin{corollary}
For $n\geq0$ let $f_n \in \EC(\lambda,\Gamma)$ and $\varphi_n(x) :\equiv - \log|\det \rmD f_n(x)|$. Then%
\begin{equation*}
  P_{m_{\infty}}(f_{\infty};\varphi_{\infty}) = P_{\tp}(f_{\infty};\varphi_{\infty}) = 0.%
\end{equation*}
\end{corollary}

\begin{proof}
The fact that $P_{\tp}(f_{\infty};\varphi_{\infty}) = 0$ is shown in the proof of Theorem \ref{thm_me_form}. The first equality immediately follows from the formula for the metric entropy and the definition of the measure-theoretic pressure.%
\end{proof}

\section{Equi-conjugacy of expanding systems}\label{sec_equiconj}%

A classical result about expanding maps, proved by Shub \cite{Shu}, asserts that any two expanding $\CC^1$-maps, defined on the same compact manifold $M$, are topologically conjugate iff their induced maps on the fundamental group $\pi_1(M)$ are algebraically conjugate. Also this result can be extended to the time-dependent situation. In particular, this will show that the full variational principle holds for a topological NDS built from expanding $\CC^1$-maps $f_n:M\rightarrow M$ which have a common expansion factor $\lambda>1$ and induce the same map on $\pi_1(M)$.%

\begin{theorem}
Let $(M,f_{\infty})$ be an NDS on a compact Riemannian manifold $M$ with $\CC^1$-expanding maps $f_n$ having expansion factors uniformly bounded away from one. Additionally assume that the map induced by $f_n$ on the fundamental group $\pi_1(M)$ is the same for all $n$, say $(f_n)_* \equiv \varphi \in \mathrm{End}(\pi_1(M))$. Then, for any $\CC^1$-expanding map $f$ with $f_* = \varphi$ there exists an equi-conjugacy $\pi_{\infty} = (\pi_n)_{n=0}^{\infty}$ between the NDS $f_{\infty}$ and $f$.%
\end{theorem}

\begin{proof}
We will obtain the equi-conjugacy as a fixed point of a contraction on an appropriately defined space of sequences. The proof proceeds in three steps.%

\emph{Step 1.} Fix an expanding $\CC^1$-map $f:M\rightarrow M$ with $f_* = \varphi$ (for instance, $f = f_1$). Let $\pi:\tilde{M}\rightarrow M$ be the universal covering of $M$. On $\tilde{M}$ we consider the lifted Riemannian metric with distance function denoted by $\tilde{d}(\cdot,\cdot)$, which makes the covering projection $\pi$ a local isometry and $\tilde{M}$ a complete Riemannian manifold. The deck transformation group of $\pi$ is the subgroup of the isometry group $\mathrm{Iso}(\tilde{M})$ given by%
\begin{equation*}
  \rmD(\pi) = \left\{ \gamma \in \mathrm{Iso}(\tilde{M})\ :\ \pi \circ \gamma = \pi \right\}.%
\end{equation*}
If $g:M\rightarrow M$ is any self-covering of $M$ and $\tilde{g}:\tilde{M}\rightarrow\tilde{M}$ a lift of $g$, i.e., $\pi \circ \tilde{g} = g \circ \pi$, then $\tilde{g}$ is invertible and induces an endomorphism of $\rmD(\pi)$ by%
\begin{equation*}
  \tilde{g}^*:\rmD(\pi) \rightarrow \rmD(\pi),\quad \gamma \mapsto \tilde{g} \circ \gamma \circ \tilde{g}^{-1}.%
\end{equation*}
We can lift each $f_n$ to an expanding diffeomorphism $\tilde{f}_n:\tilde{M} \rightarrow \tilde{M}$, and we also lift $f$ to $\tilde{f}:\tilde{M}\rightarrow\tilde{M}$. We choose these lifts in the following way. First we pick the lift $\tilde{f}$ arbitrarily. Since $\tilde{f}^{-1}$ is a contraction, it has a unique fixed point $y_0$. We put $x_0 := \pi(y_0)$ and choose for each $n$ a continuous path $\beta_n$ from $f_n(x_0)$ to $f(x_0) = x_0$. Then there exists a unique lift $\tilde{f}_n$ of $f_n$ such that $\tilde{f}_n(y_0)$ is the endpoint of the lift of $\beta_n^{-1}$ (i.e., the path $\beta_n$ traversed in the opposite direction)which starts at $y_0$. In particular, this guarantees that $\tilde{f}^* = \tilde{f}_n^*$ for all $n$ (cf.~\cite[Proof of Thm.~3]{Shu}).%

\emph{Step 2.} Let%
\begin{equation*}
  \AC := \left\{ h \in \mathrm{Homeo}(\tilde{M})\ :\ h \circ \gamma = \gamma \circ h,\ \forall \gamma \in \rmD(\pi) \right\},%
\end{equation*}
endowed with the metric%
\begin{equation*}
  d^{\pm}_{\infty}(h,i) := \sup_{y\in\tilde{M}}\tilde{d}(h(y),i(y)) + \sup_{y\in\tilde{M}}\tilde{d}(h^{-1}(y),i^{-1}(y)).%
\end{equation*}
Finiteness of $d^{\pm}_{\infty}$ follows from the fact that there exists a compact fundamental domain $K\subset\tilde{M}$ for the natural action of $\rmD(\pi)$ on $\tilde{M}$, and hence the supremum over all $y\in\tilde{M}$ reduces to the supremum over $y\in K$. The space on which our operator acts is defined by%
\begin{equation*}
  \BC := \left\{ (h_k)_{k=0}^{\infty}\ :\ h_k \in \AC \mbox{ and } h_k(y_0) = \tilde{f}_0^k(h_0(y_0)),\ \forall k \geq 0 \right\}.%
\end{equation*}
From \cite[Thm.~5]{Shu} it follows that such sequences exist, hence $\BC \neq \emptyset$. We define%
\begin{equation*}
  D_{\infty}\left( (h_k)_{k=0}^{\infty},(i_k)_{k=0}^{\infty} \right) := \sup_{k\geq0}d^{\pm}_{\infty}(h_k,i_k).%
\end{equation*}
Let $K \subset \tilde{M}$ be a compact fundamental domain for the action of $\rmD(\pi)$ with $y_0 \in K$. Since $h_k$ and $i_k$ commute with deck transformations, also $h_k(K)$ and $i_k(K)$ are compact fundamental domains, which both contain the point $\tilde{f}_0^k(h_0(y_0))$. Let $C$ be the union of all compact fundamental domains that contain $\tilde{f}_0^k(h_0(y_0))$. Then $C$ is bounded and%
\begin{equation*}
  d_{\infty}(h_k,i_k) = \sup_{y\in K}d(h_k(y),i_k(y)) \leq \mathrm{diam}C,\ \forall k\geq0.%
\end{equation*}
Together with the analogous statement for the inverse maps, it follows that $D_{\infty}$ is finite. The proof that it is a metric is trivial. The definition of $D_{\infty}$ implies that convergence in $D_{\infty}$ is equivalent to uniform convergence in every component and for the inverses. Since the equality $h_k(y_0) = \tilde{f}_0^k(h_0(y_0))$ carries over to continuous limits, it follows that $(\BC,D_{\infty})$ is a complete metric space.%

\emph{Step 3.} We define the operator%
\begin{equation*}
  \sigma:\BC \rightarrow \BC,\quad (h_k)_{k\geq0} \mapsto (\tilde{f}_k^{-1} \circ h_{k+1} \circ \tilde{f})_{k\geq0}.%
\end{equation*}
This definition makes sense, because%
\begin{equation*}
  \tilde{f}^{-1}_k(h_{k+1}(\tilde{f}(y_0))) = \tilde{f}^{-1}_k(h_{k+1}(y_0)) = \tilde{f}^{-1}_k(\tilde{f}_0^{k+1}(h_0(y_0))) = \tilde{f}_0^k(h_0(y_0))%
\end{equation*}
and $\tilde{f}^{-1} \circ h_{k+1} \circ \tilde{f} \in \AC$, following from a similarly simple computation. To show that $\sigma$ is a contraction, note that%
\begin{eqnarray*}
  d_{\infty}\left(\tilde{f}^{-1}_k \circ h_{k+1} \circ \tilde{f},\tilde{f}^{-1}_k \circ i_{k+1} \circ \tilde{f}\right) &=& d_{\infty}\left(\tilde{f}^{-1}_k \circ h_{k+1},\tilde{f}^{-1}_k \circ i_{k+1}\right)\\
	&\leq& \lambda^{-1}d_{\infty}\left(h_{k+1},i_{k+1}\right),%
\end{eqnarray*}
where $\lambda$ is a common expansion factor of the maps $f_n$. From this observation one easily derives that $\sigma$ is a contraction, and hence there is a unique sequence $(h_k)_{k=0}^{\infty}$ in $\BC$ such that $h_{k+1} \circ \tilde{f} \equiv \tilde{f}_k \circ h_k$. Since the maps $h_k$ commute with deck transformations, we can project them to homeomorphisms $\pi_k:M \rightarrow M$ such that $\pi_{k+1} \circ f \equiv f_k \circ \pi_k$. It remains to show equicontinuity. First note that there exists a constant $c>0$ with%
\begin{equation*}
  \tilde{d}\left(\tilde{f}^n(y),\tilde{f}_k^n(h_k(y))\right) = \tilde{d}\left(\tilde{f}^n(y),h_{k+n}(\tilde{f}^n(y))\right) \leq c%
\end{equation*}
for all $y\in\tilde{M}$ and $n\geq 0$. To show this, for a fixed $y\in\tilde{M}$ let $\gamma_y\in\rmD(\pi)$ be such that $z := \gamma_y(y)$ is contained in a fixed fundamental domain $K$ of $\rmD(\pi)$ with $y_0 \in K$. Then $h_k(y) = h_k(\gamma_y^{-1}(z)) = \gamma_y^{-1}(h_k(z))$, implying $\tilde{d}(h_k(y),y) = \tilde{d}(\gamma_y^{-1}(h_k(z)),\gamma_y^{-1}(z)) = \tilde{d}(h_k(z),z)$. On $K$ the functions $h_k$ are uniformly bounded (using the same argument that was used to prove $D_{\infty}<\infty$). This shows the existence of $c$. Furthermore, for all $y,z\in\tilde{M}$,%
\begin{eqnarray*}
  \tilde{d}(h_k(y),h_k(z)) &\leq& \lambda^{-n}\tilde{d}\left(\tilde{f}_k^n(h_k(y)),\tilde{f}_k^n(h_k(z))\right)\\
	                         &\leq& \lambda^{-n}\Bigl[\tilde{d}\left(\tilde{f}_k^n(h_k(y)),\tilde{f}^n(y)\right) + \tilde{d}\left(\tilde{f}^n(y),\tilde{f}^n(z)\right)\\
													     && \qquad  + \tilde{d}\left(\tilde{f}^n(z),\tilde{f}_k^n(h_k(z))\right)\Bigr]\\
					    				     &\leq& 2c\lambda^{-n} + \lambda^{-n}\tilde{d}\left(\tilde{f}^n(y),\tilde{f}^n(z)\right).%
\end{eqnarray*}
Hence, for given $\ep>0$ we can first choose $n$ large enough so that $2c\lambda^{-n} < \ep/2$. Then we choose $\delta>0$ so that $\tilde{d}(\tilde{f}^n(y),\tilde{f}^n(z)) < \ep/2$ if $\tilde{d}(y,z) < \delta$. This implies $\tilde{d}(h_k(y),h_k(z)) < \ep$ for all $k$, whenever $\tilde{d}(y,z) < \delta$, showing equicontinuity of $(h_k)_{k=0}^{\infty}$, and hence of $(\pi_k)_{k=0}^{\infty}$. For $(\pi_k^{-1})_{k=0}^{\infty}$ the proof works analogously.%
\end{proof}

\begin{remark}
In Ruelle \cite[Sec.~4]{Ru2} one finds a similar result. Here the nonautonomous system is given as a small time-dependent perturbation of a fixed Axiom A diffeomorphism $f$ around one of its basic sets $\Lambda$. In this case, one can show the existence of a time-dependent uniformly hyperbolic set such that the restriction of the nonautonomous system to this set is equi-conjugate to the restriction of $f$ to $\Lambda$.%
\end{remark}

\begin{corollary}
For any NDS $(M,f_{\infty})$ as given in the above theorem a full variational principle holds, i.e.,%
\begin{equation*}
  \sup_{\mu_{\infty}}h_{\EC_{\mathrm{M}}}(f_{\infty};\mu_{\infty}) = h_{\tp}(f_{\infty}),%
\end{equation*}
where the supremum is taken over all IMSs $\mu_{\infty}$.%
\end{corollary}

\begin{proof}
The inequality ``$\leq$'' was proved in \cite[Thm.~28]{Kaw} or follows as a special case from Theorem \ref{thm_varineq}. The equi-conjugacy $\pi_{n+1} \circ f \equiv f_n \circ \pi_n$ given by the above theorem preserves the topological entropy, i.e., $h_{\tp}(f_{\infty}) = h_{\tp}(f)$. The map $f$ satisfies the classical variational principle $h_{\tp}(f) = \sup_{\mu}h_{\mu}(f)$, the supremum taken over all $f$-invariant probability measures $\mu$. For any such measure, $\mu_n := (\pi_n)_*\mu$ defines an IMS, i.e., $(f_n)_*\mu_n \equiv \mu_{n+1}$. If $\PC$ is a finite measurable partition of $M$, then $\PC_n := \pi_n\PC$, $n\in\N_0$, defines a sequence of partitions, which is contained in the admissible class $\EC_{\mathrm{M}}(\mu_{\infty})$ (cf.~\cite[Prop.~27]{Kaw}). This implies%
\begin{equation*}
  h_{\tp}(f) = h_{\tp}(f_{\infty}) \geq \sup_{\mu} h_{\EC_{\Mis}}(f_{\infty};\mu_{\infty}) \geq \sup_{\mu} h_{\mu}(f) = h_{\tp}(f),%
\end{equation*}
completing the proof.%
\end{proof}

\section{Expansivity}\label{sec_sue}%

In this section, we introduce an analogue of the notion of \emph{positive expansivity} for autonomous systems that is the topological counterpart to the expansivity property of the differentiable systems studied in the preceding sections.%

\subsection{Preliminary notions}%

We start by introducing some intuitive but preliminary notions of expansivity of increasing strength. Recall that a continuous map $f:X \rightarrow X$ on a metric space $X$ is called \emph{positively expansive} if there exists $\delta>0$ such that $d(f^i(x),f^i(y)) < \delta$ for all $i\geq0$ implies $x=y$.%

\begin{definition}
A topological NDS $(X_{\infty},f_{\infty})$ is called%
\begin{enumerate}
\item[(i)] \emph{time-$i$-expansive} with \emph{expansivity constant} $\delta>0$ if there exists $\delta>0$ such that for all $x,y\in X_i$ the following implication holds:%
\begin{equation*}
  \sup_{n\in\N_0}d_{n+i}(f_i^n(x),f_i^n(y)) < \delta \quad \Rightarrow \quad x=y;%
\end{equation*}
\item[(ii)] \emph{all-time expansive} if it is time-$i$-expansive for every $i\geq0$;%
\item[(iii)] \emph{uniformly expansive} if it is all-time expansive with a uniform expansivity constant $\delta$ for all of the systems $(X_{i,\infty},f_{i,\infty})$, $i\geq0$.%
\end{enumerate}
\end{definition}

\begin{remark}\
\begin{itemize}
\item In the case of an autonomous system, the notions of time-$i$-expansivity, all-time expansivity and uniform expansivity all coincide and are equivalent to positive expansivity.%
\item The concept of expansivity for nonautonomous systems introduced in \cite[Def.~2.2]{TDa} is equivalent to our notion of time-$1$-expansivity. However, while we allow a time-varying but at every time instant compact state space, the state space in \cite{TDa} is stationary and not necessarily compact.%
\end{itemize}
\end{remark}

The following examples show that the converse statements to the obvious implications%
\begin{equation*}
  \mbox{uniformly expansive} \quad\Rightarrow\quad \mbox{all-time expansive} \quad\Rightarrow\quad \mbox{time-$i$-expansive}%
\end{equation*}
fail to hold, and that an NDS built from positively expansive maps in general does not have any of these properties.%

\begin{example}
In general, the properties of time-$i$-expansivity and time-$j$-expansivity for $i<j$ are not related. Consider a system $(X_{\infty},f_{\infty})$ with $X_n \equiv \rmS^1$, $f_0(z) \equiv 1$ and $f_n(z) \equiv z^2$ (the angle-doubling map) for all $n\geq1$. This system is time-$i$-expansive for all $i\geq1$, but not time-$0$-expansive. Now consider a system $(X_{\infty},f_{\infty})$, where $X_0$ is finite and $X_n = [0,1]$ for all $n\geq1$. Let $f_0:X_0 \rightarrow X_1$ be any map and $f_n \equiv f$, $n\geq1$, for an arbitrary continuous map $f:[0,1]\rightarrow[0,1]$. The resulting system $(X_{\infty},f_{\infty})$ is obviously time-$0$-expansive, since there is a minimal positive distance for any two points in $X_0$, but not time-$i$-expansive for any $i\geq1$, following from the well-known fact that $[0,1]$ does not admit a positively expansive map.%
\end{example}

\begin{example}
A trivial example of an all-time but not uniformly expansive system is given as follows. Let each $X_n$ be a space consisting of precisely two points $x_1^{(n)},x_2^{(n)}$ such that $\diam(X_n) = d_n(x_1^{(n)},x_2^{(n)}) \rightarrow 0$ monotonically. Let $f_n$ be given by $f_n(x_i^{(n)}) = x_i^{(n+1)}$ for all $n$ and $i=1,2$. Then clearly $(X_{\infty},f_{\infty})$ is time-$i$-expansive with a maximal expansivity constant equal to $\diam(X_i)$. Since $\diam(X_i)$ is decreasing to $0$, the system is not uniformly expansive.%
\end{example}

\begin{example}
Consider the NDS $(X_{\infty},f_{\infty})$ with $X_n \equiv \rmS^1$ (endowed with the standard round metric $d_n \equiv d$ such that $\diam(\rmS^1)=1$) and let $f_n(z) \equiv z^{n+2}$ for each $n\geq0$. Although each $f_n$ is positively expansive, this system is not time-$i$-expansive for any $i$. One easily shows that $f_0^n(z) = z^{(n+1)!}$. Take $z,w\in\rmS^1$ with $d(z,w) = 1/(n+1)!$. Then%
\begin{equation*}
  d(f_0^i(z),f_0^i(w)) = \frac{(i+1)!}{(n+1)!} \leq \frac{1}{n+1} \mbox{\quad for\ } i=1,2,\ldots,n-1,% 
\end{equation*}
and $f_0^n(z) = f_0^n(w)$. Hence, $(X_{\infty},f_{\infty})$ is not time-$0$-expansive, since for every $n$ we find $z^{(n)} \neq w^{(n)}$ with $d(f_0^k(z^{(n)}),f_0^k(w^{(n)})) \leq 1/n$ for all $k\geq0$. The same argument shows that $(X_{\infty},f_{\infty})$ is not time-$i$-expansive for any $i$.%
\end{example}

\begin{remark}
In Roy \cite{Ro2}, an example of two positively expansive maps $f,g:X\rightarrow X$ on a compact space $X$ is given such that the compositions $f\circ g$ and $g\circ f$ are not positively expansive. This also implies that the topological NDS $f_{\infty} = (f,g,f,g,\ldots)$ is not time-$i$-expansive for any $i$.%
\end{remark}

\begin{proposition}
Let $(X_{\infty},f_{\infty})$ be a topological NDS.%
\begin{enumerate}
\item[(i)] The properties of time-$i$-expansivity, all-time expansivity and uniform expansivity are preserved by equi-conjugacies.%
\item[(ii)] Assume that $(X_{\infty},f_{\infty})$ is time-$0$-expansive and the map $f_0$ is a surjective local homeomorphism. Then $(X_{\infty},f_{\infty})$ is time-$1$-expansive. Consequently, if all $f_n$ are surjective local homeomorphisms, time-$0$-expansivity is equivalent to all-time expansivity.%
\end{enumerate}
\end{proposition}

\begin{proof}
We leave the easy proof of (i) to the reader. To prove (ii), let $\delta$ be an expansivity constant for $(X_{\infty},f_{\infty})$. Since $f_0$ is a local homeomorphism, every $x\in X_0$ has an open neighborhood $V_x$ that is mapped homeomorphically onto an open set $W_x \subset X_1$. We choose $V_x$ such that $\cl V_x$ is contained in a bigger open neighborhood $\tilde{V}_x$ on which $f_0$ is still a homeomorphism. This implies that the local inverse maps are uniformly continuous. By surjectivity, the sets $\{W_x\}_{x\in X_0}$ form an open cover of $X_1$. Choose a finite subcover $\{W_1,\ldots,W_l\}$ and let $\rho$ be the Lebesgue number of this subcover. Let $f_{0,i}:V_i \rightarrow W_i$, $i=1,\ldots,l$, denote the corresponding local homeomorphisms. There exists a positive $\ep < \min\{\rho,\delta\}$ such that $d_1(x,y)<\ep$ for $x,y\in X_1$ implies $x,y\in W_i$ for some $i$ and $d_0(f_{0,i}^{-1}(x),f_{0,i}^{-1}(y)) < \delta$. Now consider $x,y\in X_1$ with $d_{n+1}(f_1^n(x),f_1^n(y)) < \ep$ for all $n\geq0$ and assume $x,y \in W_i$. Put $\tilde{x} := f_{0,i}^{-1}(x)$, $\tilde{y} := f_{0,i}^{-1}(y)$. This implies $d_n(f_0^n(\tilde{x}),f_0^n(\tilde{y})) < \delta$ for all $n\geq0$, and hence $\tilde{x} = \tilde{y}$ implying $x = y$. Consequently, time-$1$-expansivity holds with the expansivity constant $\ep$.%
\end{proof}

\subsection{Strong uniform expansivity}%

As it turns out, the notions of the preceding subsection are not sufficiently strong to imply analogues of the classical properties of positively expansive maps such as the existence of generators for topological entropy or the existence of equivalent metrics in which the maps $f_n$ locally uniformly expand distances. Hence, we introduce the following stronger notion.%

\begin{definition}
A topological NDS $(X_{\infty},f_{\infty})$ is called \emph{strongly uniformly expansive (s.u.e.)} if there exists a constant $\delta>0$ such that for every $\ep>0$ there is an integer $N\geq1$ satisfying%
\begin{equation}\label{eq_bbshrink}
  d_{i,N}(x,y) < \delta \quad\Rightarrow\quad d_i(x,y) < \ep%
\end{equation}
for all $i\geq0$ and $x,y\in X_i$. The constant $\delta$ is called an \emph{expansivity constant}.%
\end{definition} 

\begin{remark}
The definition says that Bowen-balls shrink to points uniformly w.r.t.~the initial time, when the order $N$ tends to infinity. The implication \eqref{eq_bbshrink} can also be written as%
\begin{equation*}
  B^N_i(x,\delta) \subset B_{\ep}(x;d_i).%
\end{equation*}
\end{remark}

\begin{remark}
A similar characterization of expansivity for time-dependent systems can be found in Roy \cite[Lem.~7]{Ro1}, where dynamical systems on fiber bundles are considered.%
\end{remark}

If the spaces $X_n$ become larger in diameter very rapidly, \emph{s.u.e.}~systems not necessarily exhibit the essential features of positively expansive maps on compact spaces, since the expansivity can just result in ``blowing up'' the space, rather than in producing complicated dynamical behavior. Hence, we need to introduce a property for the sequence $X_{\infty}$ which excludes such a behavior.%

\begin{definition}
A sequence $X_{\infty} = (X_n)_{n=0}^{\infty}$ of compact metric spaces is called \emph{uniformly totally bounded} if for every $\ep>0$ there exists an integer $m$ such that $m$ $\ep$-balls are sufficient to cover $X_n$ for each $n\geq0$.%
\end{definition}

The following proposition summarizes elementary properties of \emph{s.u.e.}~systems.%

\begin{proposition}
Assume that $(X_{\infty},f_{\infty})$ is \emph{s.u.e.}~with expansivity constant $\delta$. Then the following assertions hold:%
\begin{enumerate}
\item[(i)] $(X_{\infty},f_{\infty})$ is uniformly expansive with expansivity constant $\delta$.%
\item[(ii)] If $(Y_{\infty},g_{\infty})$ is another topological NDS that is equi-conjugate to the given one, then also $(Y_{\infty},g_{\infty})$ is \emph{s.u.e.}%
\item[(iii)] If $(X_{\infty},f_{\infty})$ is autonomous, then it is positively expansive. Conversely, any positively expansive autonomous system $(X,f)$ is \emph{s.u.e.}%
\item[(iv)] There exists $\AC_{\infty} \in \LC(X_{\infty})$, $\AC_{\infty} = (\AC_n)_{n=0}^{\infty}$, that generates the topological entropy, i.e.,%
\begin{equation*}
  h_{\tp}(f_{\infty}) = \limsup_{n\rightarrow\infty}\frac{1}{n}\log\NC\left(\bigvee_{i=0}^{n-1}f_0^{-i}\AC_i\right).%
\end{equation*}
If $X_{\infty}$ is uniformly totally bounded, then $h_{\tp}(f_{\infty})$ is finite.%
\item[(v)] If, for some $n$, $f_n$ is surjective and open, then $f_n$ is a covering map. If, additionally, $f_{\infty}$ is equicontinuous, then the number of leaves for such $f_n$ is uniformly bounded.%
\end{enumerate}
\end{proposition}

\begin{proof}
(i) Assume that two points $x,y\in X_i$ satisfy%
\begin{equation*}
  d_{i+n}(f_i^n(x),f_i^n(y)) < \delta \mbox{\quad for all\ } n\geq0.%
\end{equation*}
This is equivalent to $d_{i,N}(x,y) < \delta$ for all $N$. Hence, for every $\ep>0$ we have $d_i(x,y)<\ep$, so $x=y$.%

(ii) Assume that $(X_{\infty},f_{\infty})$ is \emph{s.u.e.}~and denote by $\pi_{\infty} = (\pi_n)_{n=0}^{\infty}$ the equi-conjugacy ($\pi_{n+1} \circ f_n = g_n \circ \pi_n$). Let $\delta>0$ be the expansivity constant for $(X_{\infty},f_{\infty})$ and choose $\tilde{\delta} = \tilde{\delta}(\delta)$ according to the uniform equicontinuity of the family $\{\pi_n^{-1}\}_{n=0}^{\infty}$. Let $\tilde{\ep}>0$ be given and choose $\ep = \ep(\tilde{\ep})$ according to the uniform equicontinuity of the family $\{\pi_n\}_{n=0}^{\infty}$. Then choose $N = N(\ep)$ according to the \emph{s.u.e.}~property of $(X_{\infty},f_{\infty})$. Assuming $d_{i,N}^Y(y_1,y_2) < \tilde{\delta}$, we obtain%
\begin{eqnarray*}
  d_{i,N}^X\left(\pi_i^{-1}(y_1),\pi_i^{-1}(y_2)\right) &=& \max_{0\leq j\leq N}d_{i+j}^X\left(f_i^j(\pi_i^{-1}(y_1)),f_i^j(\pi_i^{-1}(y_2))\right)\\
	                                                      &=& \max_{0\leq j\leq N}d_{i+j}^X\left(\pi_{i+j}^{-1}(g_i^j(y_1)),\pi_{i+j}^{-1}(g_i^j(y_2))\right) < \delta,%
\end{eqnarray*}
and hence $d_i^X(\pi_i^{-1}(y_1),\pi_i^{-1}(y_2)) < \ep$, implying $d_i^Y(y_1,y_2) < \tilde{\ep}$. This shows that $(Y_{\infty},g_{\infty})$ is \emph{s.u.e.}

(iii) It is clear that the \emph{s.u.e.}~property implies positive expansivity. Conversely, assume that $(X,f)$ is positively expansive with expansivity constant $\delta$ and suppose to the contrary that there is $\ep>0$ such that for every $N$ there are $x_N,y_N \in X$ with%
\begin{equation*}
  \max_{0\leq j \leq N}d(f^j(x_N),f^j(y_N)) < \frac{\delta}{2} \mbox{\quad and\quad } d(x_N,y_N) \geq \ep.%
\end{equation*}
We may assume $x_N \rightarrow x$ and $y_N \rightarrow y$ in the compact space $X$. Then%
\begin{equation*}
  \sup_{j \in \N_0} d(f^i(x),f^j(y)) \leq \frac{\delta}{2} < \delta \mbox{\quad and\quad } d(x,y) \geq \ep,%
\end{equation*}
contradicting positive expansivity with expansivity constant $\delta$.%

(iv) Let $\delta$ be the expansivity constant and $\AC_n$ be the family of all open $\delta/2$-balls in $X_n$. Pick an arbitrary $\UC_{\infty} \in \LC(X_{\infty})$ and let $\rho>0$ be a common lower bound for its associated Lebesgue numbers. We choose $n$ with%
\begin{equation*}
  d_{i,n}(x,y) < \delta \quad\Rightarrow\quad d_i(x,y) < \rho.%
\end{equation*}
If $x,y \in \bigcap_{j=0}^n f_i^{-j}\AC_{i+j}$, then $d_{i+j}(f_i^j(x),f_i^j(y)) < \delta$ for $i=0,\ldots,n$, implying that $\AC^{\langle n\rangle}_i = \bigvee_{j=0}^n f_i^{-j}\AC_{i+j}$ is a refinement of the family $\BC^n_i$ of all Bowen-balls of order $n$ and radius $\delta$ in $X_i$. Moreover, by the choice of $n$, $\BC^n_i$ is finer than $\UC_i$ for every $i\geq0$, implying%
\begin{equation*}
  h_{\tp}(f_{\infty};\UC_{\infty}) \leq h_{\tp}(f_{\infty};\BC_{\infty}^n) 
	            \leq h_{\tp}(f_{\infty};\AC_{\infty}^{\langle n\rangle}) = h_{\tp}(f_{\infty};\AC_{\infty}),%
\end{equation*}
where the last equality is easy to see. This proves the first assertion. Now assume that $X_{\infty}$ is uniformly totally bounded. Then we can choose a generator $\AC_{\infty}$ such that $\AC_n$ consists of $m$ $\delta/2$-balls for each $n$ ($m = m(\delta)$). This easily implies $h_{\tp}(f_{\infty}) \leq \log m$.%

(v) From expansivity it follows that the maps $f_n$ are locally injective. Then, if $f_n$ is additionally open and onto, it is a local homeomorphism. This easily implies that $f_n$ is a covering map. We omit the details of the proof.%
\end{proof}

The following example shows that a uniformly expansive system is not necessarily \emph{s.u.e.}, even if the sequence $X_{\infty}$ is stationary.%

\begin{example}\label{ex_sue_unifexp}
For any map $f$ and $n\in\N$, we write $(f)_n$ for the finite sequence $(f,f,\ldots,f)$ of length $n$. We let $f(z) \equiv z^2$, $f:\rmS^1 \rightarrow \rmS^1$, be the angle-doubling map on the unit circle and consider the NDS $(X_{\infty},f_{\infty})$ defined by%
\begin{equation*}
  X_n :\equiv \rmS^1,\quad f_{\infty} := ( (\id_{\rmS^1})_1, (f)_1, (\id_{\rmS^1})_2, (f)_2, (\id_{\rmS^1})_3, (f)_3, \ldots ).%
\end{equation*}
If we consider on $\rmS^1$ the standard round metric with $\diam(\rmS^1) = 1$, this system is uniformly expansive with expansivity constant $1/2$, since for any two points $z,w \in \rmS^1$ with distance smaller than $1/2$, the application of $f_0,f_1,f_2,\ldots$ will finally double the angle sufficiently many times so that $d(f_0^n(z),f_0^n(w)) > 1/2$. However, the system is not \emph{s.u.e.}, because for every $\delta>0$,%
\begin{equation*}
  B_{2(1+2+\cdots+k)}^n(x,\delta) = B_{\delta}(x) \mbox{\quad for all\ } x \in X_{2(1+2+\cdots+k)},\ 1 \leq n \leq k+1.%
\end{equation*}
Hence, for a given $\ep \in (0,\delta)$ no uniform $N$ exists so that all Bowen-balls of order $N$ are contained in an $\ep$-ball.%
\end{example}

The next proposition shows that expanding systems are \emph{s.u.e.}%

\begin{proposition}
Every NDS $(M,f_{\infty})$ with $f_n \in \EC(\lambda,\Gamma)$ is \emph{s.u.e.}%
\end{proposition}

\begin{proof}
By Lemma \ref{lem2}, there exists $\delta>0$ so that every $\delta$-ball in $M$ is evenly covered by each $f_n$ and every branch of the inverse map is a contraction with uniform contraction constant $\mu\in(0,1)$. Moreover, the proof of the volume lemma shows that each Bowen-ball $B^n_i(x,\delta)$ is contained in a set of the form%
\begin{equation*}
  (f_i)_x^{-1} \circ \cdots \circ (f_{i+n-1})_{f_i^{n-1}(x)}^{-1}(B_{\delta}(f_i^n(x))).%
\end{equation*}
Hence, if $d_{i,n}(x,y) < \delta$, then $d(x,y) < \mu^{-n}\delta$. Choosing $n=n(\ep)$ with $\mu^{-n} \leq \ep/\delta$ yields the assertion.%
\end{proof}

It is well-known that a homeomorphism of a compact space $X$ is positively expansive iff $X$ is finite (see \cite{RWi} for an elementary proof). For an \emph{s.u.e.}~system, in general, an analogous result does not hold. An example is constructed as follows. Let $A:\R^n \rightarrow \R^n$ be a linear map all of whose eigenvalues have moduli greater than $1$. Let $X_0$ be the compact unit ball in $\R^d$ and put $X_n := A^nX_0$ for all $n\geq1$ (endowed with the restriction of the standard Euclidean metric). Then $f_n := A|_{X_n}:X_n\rightarrow X_{n+1}$ defines an NDS, which is \emph{s.u.e.}, and every $f_n$ is a homeomorphism. However, if we assume that $X_{\infty}$ is uniformly totally bounded, we can prove an analogous result.%

We define the diameter of a cover $\UC$ of a compact metric space as%
\begin{equation*}
  \diam(\UC) := \sup_{U\in\UC}\diam(U).%
\end{equation*}
The following lemma is a generalization of \cite[Thm., p.~316]{Aea}.%

\begin{lemma}\label{lem_diam}
Let $(X_{\infty},f_{\infty})$ be a topological NDS such that $X_{\infty}$ is uniformly totally bounded and $\#X_0=\infty$. Then for every $\UC_{\infty}\in\LC(X_{\infty})$, $\diam\left(\bigvee_{i=1}^n f_k^{-i}\UC_{k+i}\right)$ is not converging to zero uniformly in $k$.%
\end{lemma}

\begin{proof}
Let $\delta$ be a common lower bound for the Lebesgue numbers of $\UC_{\infty}$. Suppose to the contrary that $\diam(\bigvee_{i=1}^n f_k^{-i}\UC_{k+i}) \rightarrow 0$ for $n\rightarrow\infty$ uniformly in $k$. There exists $N\in\N$ such that for all $n\geq N$ and $k\geq0$, $\diam(\bigvee_{i=1}^n f_k^{-i}\UC_{k+i})<\delta$. Then $\UC_k$ is coarser than $\bigvee_{i=1}^n f_k^{-i}\UC_{k+i}$, $n\geq N$, $k\geq0$, implying%
\begin{eqnarray*}
  \NC\left(\bigvee_{i=0}^n f_0^{-i}\UC_i\right) &=& \NC\left(\bigvee_{i=1}^n f_0^{-i}\UC_i\right) = \NC\left(\bigvee_{i=0}^{n-1} f_0^{-(i+1)}\UC_{i+1}\right)\\
																								&=& \NC\left(f_0^{-1}\bigvee_{i=0}^{n-1}f_1^{-i}\UC_{i+1}\right) \leq \NC\left(\bigvee_{i=0}^{n-1}f_1^{-i}\UC_{i+1}\right).%
\end{eqnarray*}
By induction, one shows that for all $n\geq N$,%
\begin{equation*}
  \NC\left(\bigvee_{i=0}^n f_0^{-i}\UC_i\right) \leq \NC\left(\bigvee_{i=0}^Nf_{n-N}^{-i}\UC_{n-N+i}\right).%
\end{equation*}
We can estimate the right-hand side by%
\begin{eqnarray*}
  \NC\left(\bigvee_{i=0}^Nf_{n-N}^{-i}\UC_{n-N+i}\right) &\leq& \prod_{i=0}^N\NC\left(f_{n-N}^{-i}\UC_{n-N+i}\right)\\
	&\leq& \prod_{i=0}^N \NC\left(\UC_{n-N+i}\right) \leq \prod_{i=0}^N m(\delta) =: M,%
\end{eqnarray*}
where $m(\delta)$ denotes the number of $\delta$-balls needed to cover the spaces $X_k$. Hence,%
\begin{equation*}
  \NC\left(\bigvee_{i=0}^nf_0^{-i}\UC_i\right) \leq M \mbox{\quad for all\ } n\geq N.%
\end{equation*}
Choose $M+1$ distinct points $x_1,\ldots,x_{M+1}$ in $X_0$ and let $n$ be so large that $\diam(\bigvee_{i=0}^n f_0^{-i}\UC_i) < \min_{1 \leq i < j \leq M+1}d_0(x_i,x_j)$. This is a contradiction, because to cover $\{x_1,\ldots,x_{M+1}\}$ with sets whose diameters are smaller than $\min_{1 \leq i < j \leq M+1}d_0(x_i,x_j)$ requires at least $M+1$ sets.%
\end{proof}

\begin{theorem}\label{thm_finitespace}
Assume that $(X_{\infty},f_{\infty})$ is \emph{s.u.e.}, $X_{\infty}$ is uniformly totally bounded and every $f_n$ is a homeomorphism so that the family $\{f_n^{-1}\}_{n\geq0}$ is uniformly equicontinuous. Then $X_0$ and hence every $X_n$ is finite.%
\end{theorem}

\begin{proof}
Let $\AC_n$ be the cover of $X_n$ consisting of all $\delta$-balls, where $\delta$ is an expansivity constant. Then%
\begin{equation*}
  \diam\left(\bigvee_{j=0}^{n-1}f_i^{-j}\AC_{i+j}\right) \rightarrow 0 \mbox{\quad as\ } n \rightarrow \infty%
\end{equation*}
uniformly in $i$. Since each $f_n$ is a homeomorphism, we obtain a sequence of open covers for $X_{1,\infty}$ by putting $\BC_n := f_n\AC_n$, $n\geq0$. From the equicontinuity of $\{f_n^{-1}\}$ it follows that $\BC_{\infty}\in\LC(X_{\infty})$. Then%
\begin{equation*}
  \bigvee_{j=0}^{n-1}f_i^{-j}\AC_{i+j} = \bigvee_{j=0}^{n-1}f_i^{-j}f_{i+j}^{-1}\BC_{i+j}
	                                     = \bigvee_{j=0}^{n-1}f_i^{-(j+1)}\BC_{i+j} = \bigvee_{j=1}^n f_i^{-j}\BC_{i+j-1}.%	
\end{equation*}
By Lemma \ref{lem_diam}, $\diam(\bigvee_{j=1}^n f_i^{-j}\BC_{i+j-1})$ does not converge uniformly to zero if $X_0$ is infinite. Hence, $X_0$ must be finite.% 
\end{proof}

Another classical result asserts that a positively expansive map becomes expanding in a suitably chosen metric. This was proved by Reddy \cite{Red}, using Frink's metrization lemma. We will reproduce the proof for \emph{s.u.e.}~systems.%

If $X$ is a set and $A\subset X\tm X$, we write%
\begin{equation*}
  A \circ A := \left\{(x,y) \in X\tm X\ :\ \exists z \in X \mbox{ with } (x,z) \in A \mbox{ and } (z,y) \in A\right\}.%
\end{equation*}

\begin{lemma_txt}{\bf(Frink's Metrization Lemma):}
Let $X$ be a topological space. If there is a nested sequence $(U_n)_{n\geq0}$ of open symmetric neighborhoods of the diagonal $\Delta \subset X\tm X$ such that $U_0 = X\tm X$, $\bigcap_n U_n = \Delta$ and $U_n \circ U_n \circ U_n \subset U_{n-1}$ for $n\geq1$, then there is a metric $\rho$ for $X$ such that%
\begin{equation*}
  U_n \subset \left\{(x,y)\ :\ \rho(x,y) < 1/2^n\right\} \subset U_{n-1} \mbox{\quad for all\ } n\geq1.%
\end{equation*}
\end{lemma_txt}

If $A$ is a closed subset of a compact metric space, we denote by $N_{\ep}(A)$ the open $\ep$-neighborhood of $A$.%

\begin{theorem}\label{thm_expmetrics}
Let $(X_{\infty},f_{\infty})$ be an equicontinuous \emph{s.u.e.}~NDS such that each $f_n$ is onto. Then for each $n\geq0$ there exists a metric $\rho_n$ on $X_n$ with the following properties:%
\begin{enumerate}
\item[(i)] The maps $f_n:(X_n,\rho_n) \rightarrow (X_{n+1},\rho_{n+1})$ expand small distances uniformly in $n$, i.e., there exist $\alpha>0$ and $\beta>1$ so that  for any $n$ and $x,y\in X_n$, $\rho_n(x,y) < \alpha$ implies $\rho_{n+1}(f_n(x),f_n(y)) \geq \beta\rho_n(x,y)$.%
\item[(ii)] The metrics $\rho_n$ and $d_n$ are uniformly equivalent, i.e., the sequence of maps $\pi_n:(X_n,d_n) \rightarrow (X_n,\rho_n)$, $x\mapsto x$, is uniformly equicontinuous and the same holds for the sequence $(\pi_n^{-1})_{n\geq1}$ of inverses.%
\end{enumerate}
\end{theorem}

\begin{proof}
Let $\delta$ be an expansivity constant. Put $V_0^{(n)} := X_n \tm X_n$ and for each $k\geq0$ let%
\begin{equation*}
  V_k^{(n)} := \left\{(x,y) \in X_n \tm X_n\ :\ d_{n,k}(x,y) < \delta\right\}.%
\end{equation*}
Let $g_n := f_n \tm f_n:X_n \tm X_n \rightarrow X_{n+1} \tm X_{n+1}$. Then, for each $n$, the sequence $(V_k^{(n)})_{k\geq 0}$ is a nested sequence of open symmetric neighborhoods of $\Delta_n = \{(x,x) : x\in X_n\}$ such that $\bigcap_k V^{(n)}_k = \Delta_n$ and it is easy to see that%
\begin{equation}\label{eq_gnrel}
  g_n\left(V_k^{(n)}\right) = V_{k-1}^{(n+1)} \cap g_n\left(V_0^{(n)}\right),%
\end{equation}
where we use that $f_n$ is onto. Taking the product metric $D_n((x,y),(z,w)) = \max\{d_n(x,z),d_n(y,w)\}$ on $X_n\tm X_n$, we find that $N_{\alpha}(\Delta_n) \subset V_0^{(n)}$ for $\alpha = \delta/2$. From the \emph{s.u.e.}~property it follows that there exists an integer $N\geq1$ with%
\begin{equation*}
  V^{(n)}_N \subset N_{(1/3)\alpha}(\Delta_n) \mbox{\quad for all\ } n\geq0.%
\end{equation*}
Then%
\begin{equation*}
  V^{(n)}_N \circ V^{(n)}_N \circ V^{(n)}_N \subset V^{(n)}_0.%
\end{equation*}
Let $U^{(n)}_0 := X_n \tm X_n$ and for $k\geq1$, $U^{(n)}_k := V_{(k-1)N}^{(n)}$. We want to apply Frink's metrization lemma to each sequence $(U^{(n)}_k)_{k\geq0}$. To this end, it suffices to prove%
\begin{equation}\label{eq_frinkcond}
  U^{(n)}_{k+1} \circ U^{(n)}_{k+1} \circ U^{(n)}_{k+1} \subset U^{(n)}_k \mbox{\quad for all\ } k\geq0,\ n\geq0.%
\end{equation}
For $k=0,1$ this relation holds by construction. Let $p = (x,y) \in U^{(n)}_{k+1} \circ U^{(n)}_{k+1} \circ U^{(n)}_{k+1}$ for $k\geq2$. Then there exist points $a,b\in X_n$ such that $\{(x,a),(a,b),(b,y)\} \subset U^{(n)}_{k+1}$. Iterating \eqref{eq_gnrel}, for $0\leq j\leq (k-1)N$ we obtain%
\begin{eqnarray*}
  g^j_n(p) &=& (f_n^j(x),f_n^j(y)) = (f_n^j(x),f_n^j(a)) \circ (f_n^j(a),f_n^j(b)) \circ (f_n^j(b),f_n^j(y))\\
	         &\in& g_n^j(U^{(n)}_{k+1}) \circ g_n^j(U^{(n)}_{k+1}) \circ g_n^j(U^{(n)}_{k+1})\\
					 &\subset& V_{kN - j}^{(n+j)} \circ V_{kN - j}^{(n+j)} \circ V_{kN - j}^{(n+j)}\\
					 &\subset& V_1^{(n+j)} \circ V_1^{(n+j)} \circ V_1^{(n+j)} \subset V_0^{(n+j)}.%
\end{eqnarray*}
Therefore, $d_{n+j}(f_n^j(x),f_n^j(y)) < \delta$ for $j=0,\ldots,(k-1)N$, and hence%
\begin{equation*}
  p = (x,y) \in V^{(n)}_{(k-1)N} = U^{(n)}_k,%
\end{equation*}
concluding the proof of \eqref{eq_frinkcond}. Let $\rho_n$ denote the metric on $X_n$ whose existence is guaranteed by Frink's metrization lemma. Fix $n\geq0$ and let $x,y\in X_n$ with $0 < \rho_n(x,y) < 1/32$. Since the sets $U_k^{(n)}\backslash U_{k+1}^{(n)}$, $k\geq 0$, form a partition of $X_n \tm X_n$, there exists $k\geq-1$ such that $(x,y) \in U^{(n)}_{k+1}\backslash U^{(n)}_{k+2}$. Then%
\begin{equation*}
  1/2^{k+3} \leq \rho_n(x,y) < \min\left\{1/32,1/2^{k+1}\right\},%
\end{equation*}
which implies $k\geq 3$. Since $(x,y) \in U^{(n)}_{k+1} \backslash U^{(n)}_{k+2} = V_{kN}^{(n)} \backslash V_{(k+1)N}^{(n)}$, there exists $j$ with%
\begin{equation*}
  kN < j \leq (k+1)N \mbox{\quad and\quad } d_{n+j}(f_n^j(x),f_n^j(y)) > \delta.%
\end{equation*}
Let $(z,w) = (f^{3N}_n(x),f^{3N}_n(y))$. Then%
\begin{equation*}
  0 \leq (k-3)N < j - 3N \leq (k-2)N%
\end{equation*}
and%
\begin{equation*}
  d_{n+j}\left(f^{j - 3N}_{n+3N}(z),f^{j - 3N}_{n+3N}(w)\right) = d_{n+j}(f_n^j(x),f_n^j(y)) > \delta.%
\end{equation*}
Hence, $(f_n^{3N}(x),f_n^{3N}(y)) \notin V^{(n+3N)}_{(k-2)N} = U^{(n+3N)}_{k-1}$, implying%
\begin{equation*}
  \rho_{n+3N}\left(f^{3N}_n(x),f^{3N}_n(y)\right) \geq 1/2^k > 2\rho_n(x,y).% 
\end{equation*}
Thus, $f_n^{3N}$ expands small distances for each $n$. We will construct different metrics $\rho_n'$ such that the maps $f_n$ are uniformly expanding small distances w.r.t.~these metrics. But first we prove that $\rho_n$ and $d_n$ are uniformly equivalent in $n$. By construction, it holds that%
\begin{equation*}
  B_n^{(k-1)N}(x,\delta) \subset B_{1/2^k}(x;\rho_n) \subset B_n^{(k-2)N}(x,\delta)%
\end{equation*}
for each $x\in X_n$ and $k\geq0$, where the Bowen-balls are defined in terms of the given metric $d_n$. By the \emph{s.u.e.}~property, we can choose, for each $\ep>0$, $k = k(\ep)$ large enough so that $B_n^{kN}(x,\delta) \subset B_{\ep}(x;d_n)$ (for all $n$), and hence%
\begin{equation*}
  B_{1/2^k}(x;\rho_n) \subset B_{\ep}(x;d_n).%
\end{equation*}
This shows uniform equicontinuity of the sequence $\pi_n^{-1}:(X_n,\rho_n) \rightarrow (X_n,d_n)$. On the other hand, for any integer $k$, we can choose $\ep>0$ small enough so that $B_{\ep}(x;d_n) \subset B_n^{(k-1)N}(x,\delta)$ for all $n$ and $x\in X_n$, which follows from uniform equicontinuity of $f_{\infty}$. This proves uniform equicontinuity of the sequence $\pi_n:(X_n,d_n) \rightarrow (X_n,\rho_n)$.%

Now we define new metrics by%
\begin{equation*}
  \rho_n'(x,y) := \sum_{i=0}^{3N-1}\frac{1}{\mu^i}\rho_{n+i}(f_n^i(x),f_n^i(y)),%
\end{equation*}
where $\mu := 2^{1/(3N)}$. This clearly defines a metric on $X_n$ compatible with the topology, and $\rho_n'(x,y) < 1/32$ implies $\rho_n(x,y) < 1/32$, and hence%
\begin{eqnarray*}
  \rho_{n+1}'(f_n(x),f_n(y)) &=& \sum_{i=0}^{3N-1}\frac{1}{\mu^i}\rho_{n+1+i}(f_n^{i+1}(x),f_n^{i+1}(y))\\
	                           &>& \mu \sum_{i=1}^{3N-1}\frac{1}{\mu^i}\rho_{n+i}(f_n^i(x),f_n^i(y)) + \frac{2}{\mu^{3N-1}}\rho_n(x,y)\\
														 &=& \mu \sum_{i=0}^{3N-1}\frac{1}{\mu^i}\rho_{n+i}(f_n^i(x),f_n^i(y)) = \mu \rho_n'(x,y).%
\end{eqnarray*}
This gives the desired metrics in which each $f_n$ expands small distances. To complete the proof, it remains to show that $\rho_n$ and $\rho_n'$ are uniformly equivalent. This follows from the inequalities%
\begin{equation*}
  \rho_n(x,y) \leq \rho_n'(x,y) \leq 3N\rho_{n,3N-1}(x,y),%
\end{equation*}
and the uniform equicontinuity of $f_{\infty}$.%
\end{proof}

It can easily be shown that the above theorem does not hold for the uniformly expansive system in Example \ref{ex_sue_unifexp}. We consider this fact as a justification that the notion of strong uniform expansivity is the appropriate analogue of positive expansivity and postpone further investigations to a future work.%

\section*{Acknowledgments}%

The author thanks Jose C\'anovas, Yuri Latushkin and Lai-Sang Young for numerous fruitful discussions on nonautonomous dynamical systems and Anne Gr\"{u}nzig for proof-reading the manuscript. This work was supported by DFG fellowship KA 3893/1-1.%

\end{document}